\documentclass[11pt]{amsart}
\usepackage{amsmath, amsthm,amssymb, enumerate,esint}
\usepackage{mathtools}
\usepackage[dvipsnames]{xcolor}
\usepackage{hyperref}
\usepackage{enumerate}
\usepackage[pdftex]{graphicx}
\usepackage{caption}
\usepackage{subcaption}
\usepackage[margin=1in]{geometry}
\usepackage{verbatim}

\numberwithin{equation}{section}

\newtheorem{theorem}{Theorem}[section]
\newtheorem{definition}[theorem]{Definition}
\newtheorem{proposition}[theorem]{Proposition}

\newtheorem{lemma}[theorem]{Lemma}
\newtheorem{corollary}[theorem]{Corollary}

\newtheorem{example}[theorem]{Example}
\newtheorem{remark}[theorem]{Remark}

\newtheorem*{proposition*}{Proposition}

\newcommand{\R}{\mathbb{R}}  
\newcommand{\Z}{\mathbb{Z}}

\newcommand{\N}{\mathbb{N}}
\newcommand{\e}{\epsilon}
\newcommand{\hk}{\mathsf{h}}

\def\cC{\mathcal C}
\def\ssu{\subset}
\def\rr{\mathbb R}

\newcommand{\tiling}{\mathsf{t}}
\newcommand{\Zptn}{Z}
\newcommand{\numtilings}{N}
\newcommand{\numlozenge}{\mathcal{L}}
\DeclareMathOperator{\Lip}{Lip}
\DeclareMathOperator{\weight}{wt}
\DeclareMathOperator{\hookwts}{hooks}

\DeclareMathOperator{\Lob}{\Lambda}
\def\area{{\textrm{area}}}

\usepackage{url}
\usepackage{hyperref}

\newcommand{\SYT}{\operatorname{SYT}}
\newcommand{\Hex}{\operatorname{\mathbb H}}

\def\emp{\varnothing}

\def\rr{\mathbb R}

\def\la{\lambda}

\def\al{\alpha}
\def\be{\beta}

\def\cC{\mathcal C}

\def\ssu{\subset}

\def\wt{\widetilde}
\def\<{\langle}
\def\>{\rangle}

\def\Ups{\Upsilon}

\def\0{{\mathbf 0}}

\def\PP{\textup{\textsf{P}}}

\def\PP{{\textup{\textsf{P}}}}

\newcommand{\ED}{\mathcal{E}}

\def\.{\hskip.06cm}
\def\ts{\hskip.03cm}

\def\nin{\noindent}

\begin{document}

	\title[Standard Young tableaux and weighted lozenge tilings]
    {Asymptotics for the number of standard tableaux of skew shape and for weighted lozenge tilings}
	
	\author{Alejandro H. Morales}
	\address{Department of Mathematics and Statistics, UMass, Amherst, MA}
	\email{amorales@math.umass.edu}
	
	\author{Igor Pak}
	\address{Department of Mathematics, University of California, Los
		Angeles, CA}
	\email{pak@math.ucla.edu}

	\author{Martin Tassy}
	\address{Department of Mathematics, Dartmouth College, Hanover, NH}
	\email{mtassy@math.dartmouth.edu}
	
\date{\today}	
	
\maketitle

\begin{abstract}
We prove and generalize a conjecture in~\cite{MPP4} about the asymptotics
of \ts $\frac{1}{\sqrt{n!}} f^{\la/\mu}$, where $f^{\la/\mu}$ is the number
of standard Young tableaux of skew shape $\la/\mu$ which have stable
limit shape under the \ts $1/\sqrt{n}$ \ts scaling.  The proof is based
on the variational principle on the partition function of certain weighted
lozenge tilings.
\end{abstract}

\vskip.7cm
	
\section{Introduction}

In enumerative and algebraic combinatorics, \emph{Young tableaux} are
fundamental objects that have been studied for over a century
with a remarkable variety of both results and applications to other fields.
The asymptotic study of the number of standard Young tableaux is an
interesting area in its own right, motivated by both probabilistic
combinatorics (\emph{longest increasing subsequences}) and
representation theory.  This paper is a surprising new advance in this
direction, representing a progress which until recently could not be
obtained by existing tools.

\medskip

\subsection{Main results}
Let us begin by telling the story behind this paper.  Denote by
$f^{\la/\mu}=\SYT(\la/\mu)$ the number of standard Young tableaux of skew
shape~$\la/\mu$.  There is \emph{Feit's determinant formula} for $f^{\la/\mu}$,
which can also be derived from the Jacobi--Trudy identity for skew shapes.
In some cases there are multiplicative formulas for $f^{\la/\mu}$,
e.g.\  the \emph{hook-length formula} (HLF) when $\mu=\emp$, see also~\cite{MPP3}.
However, in general it is difficult to use Feit's formula to obtain even
the first order of asymptotics, since there is no easy way to diagonalize the
corresponding matrices.

It is easy to show (see e.g.~\cite{Pak-skew}), that when
\ts $|\la/\mu|=N$ and $\la_1,\ell(\la) \ts \le \ts s \sqrt{N}$, we have:
$$
c_1^N \. \le \. \frac{\bigl(f^{\la/\mu}\bigr)^2}{N!} \. \le \. c_2^N,
$$
where $c_1,c_2>0$ are universal constants which depend only on~$s$.
Improving upon these estimates is of interest in both
combinatorics and applications (cf.~\cite{MPP3,MPP4}).

In~\cite{MPP4}, much sharper bounds on $c_1,c_2$ were given,
when the diagrams $\la$ and $\mu$ have a limit shape $\psi/\phi$ under
$1/\sqrt{N}$ scaling in both directions (see below).
Based on observations in special cases, we conjectured that
there is always a limit
$$
\lim_{N\to \infty} \. \frac1N \. \log \frac{\bigl(f^{\la/\mu}\bigr)^2}{N!}
$$
in this setting.  The main result of this paper is a proof of this conjecture.

\begin{theorem} \label{thm:constantskew}
Let $\bigl\{\lambda^{(N)}\bigr\}$ and $\bigl\{\mu^{(N)}\bigr\}$
be two partition sequences with strongly stable \ts $($limit$)$ \ts shapes $\psi$
and $\phi$, respectively {\rm (see~$\S$\ref{def:stronglystableshape} for precise
definitions)}.  Let $\nu^{(N)} := \lambda^{(N)}/\mu^{(N)}$, such that
\ts $|\nu^{(N)}| = N + o(N/\log N)$.  Then
\[
\frac{1}{N} \left(\log f^{\nu^{(N)}} \. - \. \frac{1}{2}\ts N\ts\log N\right)
\. \longrightarrow \. c(\psi/\phi) \quad \text{as} \ \ \, N\to \infty,
\]
for some fixed constant \ts $c(\psi/\phi)$.
\end{theorem}

The constant $c(\psi/\phi)$ is given in Corollary~\ref{cor:explicit-constantskew}. 
The proof of the theorem is even more interesting perhaps than one would
expect.  In~\cite{Nar}, Naruse developed a novel approach to counting $f^{\la/\mu}$,
via what is now known as the \emph{Naruse hook-length formula} (NHLF):
\begin{equation} \label{eq:Naruse}
f^{\lambda/\mu} \,  = \, N! \, \sum_{D \in \ED(\lambda/\mu)}\,\.\.
 \prod_{u \in \lambda\setminus D} \. \frac{1}{\hk_{\lambda}(u)}\ts\.,
\end{equation}
where \ts $\ED(\lambda/\mu)\subseteq \binom{[\la]}{|\mu|}$ \ts is a collection of certain
subsets of the Young diagram~$[\la]$ called \emph{excited diagram} (see e.g. \cite[Sec. 2.6]{MPP3}), and $\hk_{\lambda}(u)$ is the hook-length at $u \in \la$.
The (usual) hook-length formula is a special case $\mu=\emp$.
Let us mention that \ts $\ED(\lambda/\mu)$ \ts can be viewed as the set
of certain particle configurations, giving it additional structure~\cite{MPP3}.

Although \ts $\ED(\lambda/\mu)$ \ts can have exponential size, the NHLF can
be useful in getting the asymptotic bounds~\cite{MPP4}.  It has been reproved
and studied further in~\cite{MPP1,MPP2,Kon,NO}, including the $q$-analogues
and generalizations to trees and shifted shapes. See~$\S$\ref{ss:background-naruse}
for the precise statements.

The next logical step was made in~\cite{MPP3}, where a bijection between
\ts $\ED(\lambda/\mu)$ \ts and lozenge tilings of a certain region was
constructed.  Thus, the number of standard Young tableaux \ts $f^{\la/\mu}$ \ts
can be viewed as a statistical sum of weighted lozenge tilings.  In a special
case of \emph{thick hooks} this connection is especially interesting, as the
corresponding weighted lozenge tilings were previously studied in~\cite{BGR}
(see the example below).

Now, there is a large literature on random lozenge tilings of the hexagon
and its relatives in connection with the \emph{arctic circle} phenomenon,
see \cite{CEP96,CKP01,Ken09}. In this paper we adapt the
\emph{variational principle} approach in these papers to obtain the
arctic circle behavior for the weighted tilings as well.  Putting all
these pieces together implies Theorem~\ref{thm:constantskew}.

Let us emphasize that the approach in this paper can be used
to obtain certain probabilistic information on random SYTs of
large shapes, e.g.\ in~\cite[$\S$8]{MPP3} we show how to compute
asymptotics of various path probabilities.  However, in the absence
of a direct bijective proof of NHLF, our approach cannot be easily adapted
to obtain limit shapes of SYTs as Sun has done recently~\cite{Sun}
using the \emph{beads model}, and Gordenko~\cite{Gor} using a modified
\emph{totally asymmetric simple exclusion process} (TASEP);
see also~$\S$\ref{ss:finrem-var}.

\medskip

\subsection{Thick hooks}\label{ss:intro-hooks}
Let \ts $\la = (a+c)^{b+c}$, \ts $\mu = a^b$, \ts $N=|\la/\mu|=c\ts (a+b+c)$,
where $a,b,c\ge 0$.  This shape is called the \emph{thick hook} in~\cite{MPP4}.
The HLF applied to the 180 degree rotation of $\la/\mu$ gives:
$$
f^{\la/\mu} \, = \, N! \, \. \frac{\Phi(a)\,\Phi(b)\,\Phi(c)^2
\,\Phi(a+b+c)^2}{\Phi(a+b)\,\Phi(a+c)\,\Phi(b+c)\.\Phi(a+b+2c)}\,.
$$
Here the {\em superfactorial} \ts $\Phi(n)\ts =\ts 1!\cdot 2! \ts \cdots \ts (n-1)!$ \ts is
the integer value of the \emph{Barnes $G$-function}, see e.g.~\cite{AsR}.

On the other hand, $\ED(\la/\mu)$ in this case in bijection with the set of lozenge tilings
of the hexagon \ts $\Hex(a,b,c) = \<\ts a\times b \times c \times a \times b \times c\ts\>$,
and the weight is simply a product of a linear function on horizontal lozenges (see below).
The number of lozenge tilings in this case is famously counted by the \emph{MacMahon box formula}
for the number  $\PP(a,b,c)$ of \emph{solid partitions} which fit into a
\ts $[a\times b \times c]$ \ts box:
$$
\bigl|\ED(\la/\mu)\bigr| \, = \, \PP(a,b,c) \, = \, \frac{\Phi(a)\,\Phi(b)\,\Phi(c)\, 
\Phi(a+b+c)}{\Phi(a+b)\,\Phi(b+c)\,\Phi(a+c)}\,,
$$
see e.g.~\cite[$\S$7.21]{Stanley-EC}.  It was noticed by Rains (see~\cite[$\S$9.5]{MPP3}),
that in this example our weights are special cases of multiparameter weights studied
in~\cite{BGR} in connection with closed formulas for \emph{$q$-Racah polynomials},
cf.~$\S$\ref{ss:finrem-racah}.

Now, Theorem~\ref{thm:constantskew} in this case does not give anything new, of course,
as existence of the limit when \ts $c\to \infty$, $a/c \to \al$ and $b/c\to \be$,
follows from either the Vershik--Kerov--Logan--Shepp \emph{hook integral} of the
strongly stable shapes~\cite[$\S$6.2]{MPP4} (see also~\cite{Rom}), or from the
asymptotics of the superfactorial:
$$
\log \Phi(n) \,  = \,  \frac{1}{2} \. n^2 \ts \log n \. - \. \frac34 \. n^2 \. +
\. 2\ts n \ts \log n \. + \. O(n)\ts.
$$
This gives the exact value $c(\psi/\phi)$ as an elementary function of $(\al,\be)$.

\medskip

\subsection{Thick ribbons} \label{ss:intro-ribbons}
Let $\nu_k:=(2k-1,2k-2,\ldots,2,1)/(k-1,k-2,\ldots,2,1)$.
This skew shape is a strongly stable shape. The main theorem implies that there is a limit
$$
\frac{1}{N}\ts \left(\log f^{\nu_k} \. - \. \frac{1}{2}\. N \log N \right) \, \to \, C
\quad \text{as} \ \, k \to \infty,
$$
where $N=|\nu_k| = k(3k-1)/2$.
This proves a conjecture in \cite[$\S$13.7]{MPP4}. The best known upper and lower
bounds for~$C$ calculated in \cite{Pak-skew} are
$$
-0.2368 \, \le \, C \, \le \, -0.1648\ts,
$$
but the exact value of~$C$ has no known closed formula.  This paper describes~$C$
as solution of a certain very involved variational problem
(see Corollary~\ref{cor:explicit-constantskew} and~$\S$\ref{ss:finrem-var}).

\medskip

\subsection{Structure of the paper}
We start with Section~\ref{s:background} which reviews the notation and
known results on tilings, standard Young tableaux and limit shapes.
In Section~\ref{s_main_result} we state our main technical result
(Theorem~\ref{thm:main}) on the variational principle for weighted
lozenge tilings, whose proof is postponed until Section~\ref{sec:proofwtvarprinciple}.
In the technical Section~\ref{s:count} we deduce Theorem~\ref{thm:constantskew}
from the variational principle.  We conclude with final remarks and open
problems in Section~\ref{s:finrem}.

\bigskip
\section{Background and notation} \label{s:background}

\subsection{Tilings and height functions}\label{ss:background-tilings}
Let $R$ be a connected region in the triangular lattice.
One can view  a lozenge tiling of $R$ as a stepped surface in $\R^3$
where the first two coordinates are the coordinates of the points in
the lattice and the third coordinates is the height function $h(\cdot)$ of a lozenge tiling defined in the following way:
\begin{itemize}
	\item For every edge $(x ,y)$ in $R$,  $h(y)-h(x)= 1$ if $(x,y)$ is a vertical edge and $h(y)-h(x)= 0$ otherwise.
\end{itemize}
In fact, there is a one to one correspondence between tilings of a given region and functions which verify this property defined up to a constant. Using this bijection, we will denote by $\tiling_h$ the tiling associated to a given height function $h$ and  we will do most of the subsequent reasoning using height functions rather than tilings. The \emph{position} $(x,y)$ of a horizontal lozenge in $\tiling_h$ is the coordinates of its upper corner in the triangular lattice. 

\begin{figure}[hbt]
		\centering
		\includegraphics[scale=0.68]{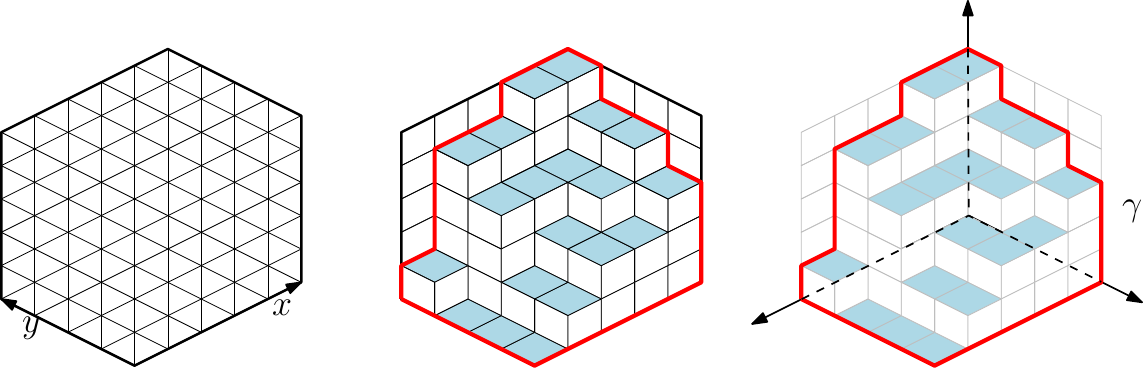}
		\caption{A region $R$ of the triangular lattice. A lozenge tiling of
			that region and the associated admissible
			stepped curve (ASC).}
		\label{fi:coloring}
	\end{figure}

We extend the definition of height functions to any region of the lattice as follows:
for general sets $S$, we say that a function $h:S \to \mathbb{Z}$ is a height function
if its restriction on each simply connected component of $S$ is a height function.

Let $R$ be a lozenge tileable region.  We say that the three dimensional curve
obtained by traveling along $\partial R$ and recording the height of each
point is an {\em admissible stepped curve} (ASC). See Figure~\ref{fi:coloring}.
	
\begin{lemma}\label{lem:extension}
Let $R$ be a connected region in the triangular grid and let $g$ be a
height function on a subset $S$ of~$R$, such that for all $x=(x_1,x_2),y = (y_1,y_2) \in S$:
\begin{equation}\label{eq:lips}	
		 g(y)-g(x) \. \leq \. \min \ts\{ y_1-x_1,y_2-x_2\}.
\end{equation}
Then $g$ can be extended into a height function on the whole region~$R$.
\end{lemma}

The lemma is a variation on~\cite[Thm.~4.1]{PST} (see also~\cite{Thu}).
It can be viewed as a Lipschitz extendability property on height
functions (cf.~\cite{CPT}). We include a quick proof for completeness.

\begin{proof}
Note that \ts $h_x(y)=g(x) + \min \{ y_1-x_1,y_2-x_2\}$ \ts is the height
function of the maximal tiling centered at $x$ and with height
$g(x)$ at $x$ (see Figure~\ref{fig:lemma21}). Define \ts
$h(y):= \min_{x\in S} \ts h_x(y)$. Since the minimum of two height functions
is still a height function, we conclude that $h$ is itself a height function.
Moreover, the inequality~\eqref{eq:lips} implies that for all pairs
\ts $x,y \in S: g(y) \leq h_x(y)$.  We conclude that $h(y)=g(y)$, which implies
the result. \end{proof}
	
\begin{figure}
\includegraphics[scale=0.57]{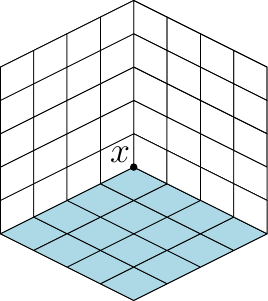}
\hskip3.cm
\includegraphics[scale=0.57]{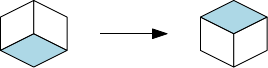}
\caption{Left: height function of the maximal tiling centered at $x$ with
  height $g(x)$. Right: the local move on lozenges.}
\label{fig:lemma21}
\end{figure}



Secondly, we need the following standard
proposition which will be useful later in this article.
	
\begin{proposition}[see~\cite{Thu}] \label{obs:lozenge2triangles}
Every two lozenge tilings of a simply connected region $R$
have equal number of lozenges of each type.
\end{proposition}

In other words, the number of lozenges of each type depends only on~$R$ and
not on the tiling.  This follows, e.g.\ since every two tilings of $R$ are
connected by local moves which do not change the number of lozenges of each
type (see Figure~\ref{fig:lemma21}).

\medskip

\subsection{Skew shapes and tableaux}\label{ss:background-skew}
Let $\lambda = (\lambda_1,\ldots,\lambda_r)$ and $\mu=(\mu_1,\ldots,\mu_s)$
denote integer partitions of length $\ell(\lambda)=r$ and $\ell(\mu)=s$.
The size of the partition is denoted by $|\lambda|$. We denote by
$\lambda'$ the \emph{conjugate partition}, and by $[\lambda]$
the corresponding \emph{Young diagram} (in English notation).
The \emph{hook length} $\hk_{\lambda}(x,y)$ of a cell $(x,y)\in\lambda$ is
defined as \ts $\hk_{\lambda}(x,y):=\lambda_x-x+\lambda'_y-y+1$.
It counts the number of cells directly to the right and directly
below $(x,y)$ in $[\lambda]$. The content of a cell $(x,y) \in \lambda$ is $c(x,y)=y-x$.

A \emph{skew shape} $\lambda/\mu$ is defined as the difference
of two shapes. Let $N = |\lambda/\mu|$.  We always assume that
the skew shape is connected. Let $N = |\lambda/\mu|$, $a=\mu_1$, $b=\ell(\mu)$ and $d_a,d_{a-1},\ldots,d_{-b}$ be the length of the diagonals in $[\lambda/\mu]$ with contents $x-y=a,a-1,\ldots,-b$, respectively. Note that the $d_{\mu_1-1},d_{\mu_2-2},\ldots,d_{\mu_b-b}$ are the diagonal lengths of the last cell $(k,\mu_k)$ of a row of $\mu$.

A \emph{standard Young tableau}
(SYT) of shape $\lambda/\mu$ is a bijective function \ts
$T: [\la/\mu] \to \{1,\dots,N\}$, increasing in rows and columns.
The number of such tableaux is denoted by $f^{\lambda/\mu}$.
This counts the number of linear extensions of the poset defined
on $[\lambda/\mu]$, with cells increasing downward and to the right.

\medskip

\subsection{Naruse's hook-length formula}\label{ss:background-naruse}
As mentioned in the introduction, the Naruse hook-length formula~\eqref{eq:Naruse}
gives a positive formula for $f^{\lambda/\mu}$. It was restated in
\cite[\S 7]{MPP3}
in terms of lozenge tilings as follows.

\begin{figure}[hbt]
\includegraphics[scale=1]{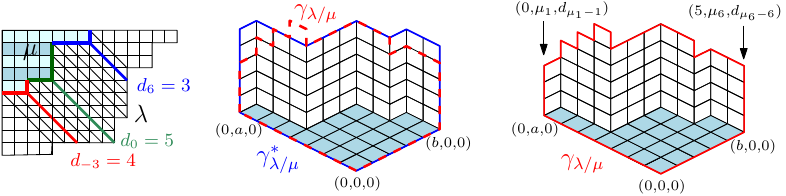}
\caption{Skew shape $\lambda/\mu$, the ASC $\gamma^*_{\lambda/\mu}$ (in blue), and the ASC $\gamma_{\lambda/\mu}$ (in red).}
\label{fig:skewshapetiling}
\end{figure}

Let $\lambda/\mu$ be a skew shape with $N$ cells and $a=\mu_1,
b=\ell(\mu)$.  Let $\gamma^*_{\lambda/\mu}$ be the ASC in
$\mathbb{R}^3$ bounded below by the polygonal chain with points $(0,a,0),(0,0,0),(b,0,0)$ and
with upper side given by the polygonal chain with points $(k,\mu_k,d)$
where $d=\max(d_{\mu_1-1},\ldots,d_{\mu_b-b})$\footnote{In \cite[\S 7.1]{MPP3} there is a
  typo since $d$ is defined there to be $\max(d_{\mu_1-1},d_{\mu_b-b})$, in our notation.}. Let $H^*_{\lambda/\mu}$ be the set of height functions $h$
that extend $\gamma^*_{\lambda/\mu}$ with the additional restriction
that on each vertical
diagonal $x-y=k$ there are no horizontal lozenges of $\tiling_h$ with position
$(x,y)$  with $y>\lambda_x$.

For our purposes, it will be more convenient to modify the upper side
of the ASC instead of imposing the above restriction on the horizontal lozenges. In order to do this, let $\gamma_{\lambda/\mu}$ be the
ASC in $\mathbb{R}^3$ bounded below by the polygonal chain with points
$(0,a,0),(0,0,0),(b,0,0)$ and
with upper side given by the polygonal chain with points
$(x_i,y_i,d_i)$ for $i=a,a-1,\ldots,b$, where $(x_i,y_i)$ is the last cell of $[\mu]$ in the diagonal with content $i$
  (see
    Figure~\ref{fig:skewshapetiling}). Let  $H_{\lambda/\mu}$ be the
    set of height functions $h$ that extend $\gamma_{\lambda/\mu}$.

For both height functions in $H^*_{\lambda/\mu}$ and
$H_{\lambda/\mu}$, the weight of a
horizontal lozenge of $\tiling_h$ with position $(x,y)$ 
is the
hook length $\hk_{\lambda}(x,y)$.
The weight of a tiling $\tiling_h$ is the product of the weights
of its horizontal lozenges and we denote it by $\hookwts_{\lambda}(\tiling_h)$,
\[
\hookwts_{\lambda}(\tiling_h) := \prod_{\lozenge \in \tiling_h} \hk_{\lambda}(x_{\lozenge},y_{\lozenge}).
\]

\begin{proposition} \label{prop: lozenge tilings MPP3 and MPT are the same}
There is a weight preserving bijection between the height functions in
$H_{\lambda/\mu}$ and $H^*_{\lambda/\mu}$.
\end{proposition}

\begin{proof}
The ASC $\gamma_{\lambda/\mu}$ and $\gamma^*_{\lambda/\mu}$ are bounded below by the same polygonal chain with points $(0,a,0),(0,0,0),(b,0,0)$. The  restriction on the horizontal lozenges of $\tiling_h$ for height functions $h$ in $H^*_{\lambda/\mu}$ is the same as the restriction imposed by the upper side of the polygonal chain bounding the ASC $\gamma_{\lambda/\mu}$. Therefore, there is a correspondence on the height functions that extend each of the ASCs. Moreover, if $h$ in $H_{\lambda/\mu}$ corresponds to $h'$ in $H^*_{\lambda/\mu}$, the horizontal lozenges in $\tiling_h$ and $\tiling_{h'}$ have the same respective positions $(x,y)$  and therefore $\hookwts_{\lambda}(\tiling_h)=\hookwts_{\lambda}(\tiling_{h'})$, as desired.
\end{proof}

\begin{remark}
    Alternatively, from the proof of \cite[Thm. 7.2]{MPP3} there is a correspondence from lozenge tilings for height functions in   $H^*_{\lambda/\mu}$ (in $H_{\lambda/\mu}$) to  \emph{excited diagrams}  in $\mathcal{E}(\lambda/\mu)$ (see e.g. \cite[Sec. 2.6]{MPP3} for a definition of  these diagram). The map is as follows, $\tiling_h$ corresponds to a \emph{reverse plane partition} of shape $\mu$: $\tiling_h \mapsto \pi$ where $\pi(i,j)$ is the height  of the horizontal lozenge at position $(i,j)$. Next, such plane partitions $\pi$ correspond to excited diagrams $D$ in $\mathcal{E}(\lambda/\mu)$ : $\pi \mapsto D$ where $r_{i,j}:=\pi_{i,j}-i$ is the row number of the final position of cell $(i,j)$ of $\mu$ after it has been moved under the excited moves from $[\mu]$ to $D$ (see Figure~\ref{fig:tiling to excited diagram}). The condition that $D\subset [\lambda]$ is equivalent to the restriction on horizontal lozenges imposed by either the height restriction on such lozenges for $h$ in $H^*_{\lambda/\mu}$ or by the polygonal chain with points $(x_i,y_i,d_i)$ where $d_i$ are the lengths of diagonals of $\lambda/\mu$ in the definition of $H_{\lambda/\mu}$.
\end{remark}

\begin{figure}
    \centering
    \includegraphics{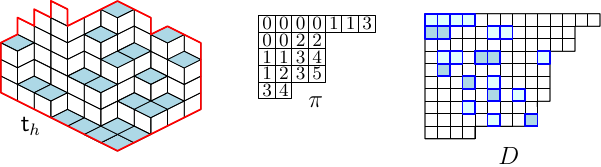}
    \caption{A lozenge tiling $\tiling_h$  for a height function $h$ in $H_{\lambda/\mu}$ and its corresponding reverse plane partition $\pi$ and excited diagram $D$. Excited diagrams appear in  Naruse's formula.}
    \label{fig:tiling to excited diagram}
\end{figure}

\begin{theorem}[{Naruse \cite{Nar}; lozenge tiling versions
    \cite[$\S$7]{MPP3}}] \label{thm:NaruseTiling0}
	\begin{align} \label{eq:NaruseTiling0}
	f^{\lambda/\mu}  &=   \frac{N!}{\prod_{(x,y) \in \lambda} \hk_{\lambda}(x,y)} \sum_{h \in H^*_{\lambda/\mu}} \hookwts_{\lambda}(\tiling_h)\\
	&= \frac{N!}{\prod_{(x,y) \in \lambda} \hk_{\lambda}(x,y)}
          \sum_{h \in H_{\lambda/\mu}} \hookwts_{\lambda}(\tiling_h). \label{eq:NaruseTiling}
	\end{align}
\end{theorem}

\begin{proof}
The first lozenge tiling reformulation is implicit in \cite[\S
7.1]{MPP3}. The second formulation follows from the first one by using Proposition~\ref{prop: lozenge tilings MPP3 and MPT are the same}.
\end{proof}

\begin{example}{\rm
The skew shape $332/21$ has five height functions that extend $\gamma_{2,1,1}$:
\begin{center}
\includegraphics[scale=0.5]{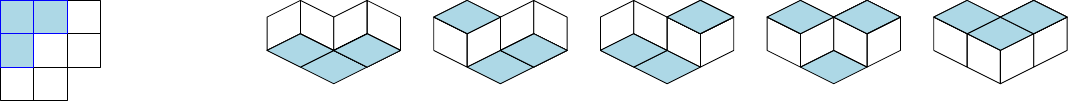}
\end{center}
Formula \eqref{eq:NaruseTiling} yields in this case
\[
f^{332/21} \, = \, \frac{5!}{5\cdot 4^2\cdot 3 \cdot 2^2} \. \bigl(5\cdot 4 \cdot 4
\ts + \ts 5\cdot 4\cdot 1 \ts + \ts  5\cdot 4\cdot 1 \ts + \ts  5 \cdot 1 \cdot 1 \ts + \ts  3\cdot 1 \cdot 1\bigr) \, = \. 16.
\]
}\end{example}

\medskip

\subsection{Stable shapes} \label{def:stronglystableshape}
Let $\psi:[0,a] \to [0,b]$ be a non-increasing left continuous function. Assume a sequence of partitions $\{\lambda^{(N)}\}$ satisfies the following property
	\[
	(\sqrt{N}-L)\ts \psi \. < \. [\lambda^{(N)}] \. < \. (\sqrt{N} + L)\ts \psi, \ \ \text{ for some } \ \. L >0,
	\]
where $[\lambda]$ denotes the function giving the boundary of
the Young diagram of $\lambda$. In this setting, we say that
$\{\lambda^{(N)}\}$ has a strongly stable shape $\psi$ and
denote it by $\lambda^{(N)} \to \psi$. Note that $\ell(\lambda^{(N)})$,
$\ell(\lambda^{(N)'}) = O(\sqrt{N})$.
Such shapes are also called {\em balanced} (see e.g.~\cite{FeS}).

Let $\psi,\phi: [0,a] \to [0,b]$ be non-increasing left continuous functions, and suppose
that $area(\psi/\phi)=1$. Let $\{ v_N = \lambda^{(N)}/\mu^{(N)}\}$ be a sequence of
skew shapes with the strongly stable shape $\psi/\phi$,
i.e. $\lambda^{(N)} \to \psi$,
$\mu^{(N)} \to \phi$ and in addition they satisfy the condition
\begin{equation} \label{eq:condition_area}
        |\mu^{(N)}| \. = \. \area(\phi) \cdot N \. +  \. o(N/\log N)\ts.
\end{equation}

Denote by $\cC(\psi), \cC(\psi/\phi) \ssu \rr_+^2$ the region below
the curve $\psi$ and  between
the curves $\psi$ and $\phi$, respectively. One can view $\cC(\psi/\phi)$ as the stable shape of the skew diagrams.

\smallskip

Finally, define the {\em hook function}
$\hbar:\mathcal{C}(\psi) \to \mathbb{R}_+$ to be the limit of the scaled function
of the hooks:
\begin{equation} \label{eq:defhbar}
\hbar(x,y) \. := \lim_{N \to \infty} \frac{1}{\sqrt{N}} \.\ts \hk_{\lambda^{(N)}}\bigl(\lfloor
x\sqrt{N}\rfloor, \lfloor y\sqrt{N}\rfloor\bigr)=\psi(x)+\psi^{-1}(y)-x-y\.,
\end{equation}
where $\psi^{-1}:[0,b] \to [0,a]$ is the inverse of the function $\psi$, which is the continuous analogue of the conjugate partition, i.e. $\lambda^{'(N)} \to \psi^{-1}$.

\section{Variational principle for weighted lozenge tilings}\label{s_main_result}
Lozenge tilings is a dimer model and the existence of a variational principle
which governs the limiting behavior of dimers under the uniform measure is a
well known result (see \cite{CKP01}). Our goal in this section will be to extend it to the case where we add weights to each tilings
that depend on the position and the type of the lozenge tiles.
	
\medskip	
	
\subsection{Weighted tilings and smooth weights}
Let $D\ssu \R^2$ be a connected domain in the plane, and let \ts
$\{w^{(i)}:D \to \R \}_{i \leq 3}$ \ts be three real valued
functions corresponding to the weight of each type of
lozenge. For a region \ts $R \subset D$, define the \textit{weight}
of a height function $h$ on $R$ associated to the weight functions
\ts $w=(w^{(1)},w^{(2)},w^{(3)})$ \ts as
\begin{equation} \label{eq:defweight-tiling} 
\weight(h) \, := \, \prod_{\lozenge \in \tiling_h} \. \exp(w^{(i_\lozenge)}(x_\lozenge,y_\lozenge)),
\end{equation}
where \ts $(x_\lozenge,y_\lozenge)$ \ts are the coordinates of the
center of the tile $\lozenge$ and $i_\lozenge \in \{1,2,3\}$
is the type of the lozenge tile:

\begin{center}
\includegraphics[scale=1.1]{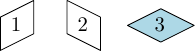}
\end{center}

\nin
Given a weight function $w$, the partition function associated to an ASC $\gamma$
is defined as:
\[
\Zptn(\gamma,w) \. := \, \sum_{h \in H_{\gamma}} \. \weight(h),
\]
where $H_{\gamma}$ is the set of height functions which extend~$\gamma$.
Let $\numtilings_{\gamma}$ be the size of $H_{\gamma}$ and let
$\numlozenge^{(i)}(\gamma)$ be the (common) number of type~$i$
  lozenges in each height function that extends $\gamma$.

\begin{definition} \label{def:convweight}
	Let $D$ be a domain in $\mathbb{R}^2$. A sequence of weight functions $\{w_n\}_{n \in \N}$  converges
        to a piecewise smooth function $\rho: D \to \R^3$ if it has
        the following property:
\begin{equation} \label{prop:defnweights}
        \tag{$\ast$}
        \lim_{n\to \infty} \sup_{(x_1,x_2) \in D} \| w_n(nx_1,nx_2) - \rho(x_1,x_2) \|_{\infty} \. = \. 0.
\end{equation}
\end{definition}

\bigskip

\subsection{The variational principle}
Our goal in this section is to establish a variational
principle for weighted tilings. We recall the unweighted version of
the variational principle from \cite[Thm. 9]{Ken09}. Let
$\Lip_{[0,1]}$ be the set of $1$-Lipschitz functions
$f:\R^2 \to \R$ that satisfy
\begin{equation} \label{eq:defLipschitz}
0\, \leq \, \partial_{x_1}f,\,\,\,\,\,\.\partial_{x_2}f,\,\,\,\,\,\. 1-\partial_{x_1}f-\partial_{x_2}f \, \leq\, 1
\end{equation}
everywhere except on a set of Lebesgue measure~$0$.  Let
\begin{equation}\label{eq:Lob}
\sigma(s,t) \, := \, \frac{1}{\pi}\ts \biggl(\Lob(\pi s) \. + \. \Lob(\pi t) \. + \. \Lob\bigl(\pi (1-s-t)\bigr)\biggr),
\end{equation}
where $\Lob(\cdot)$ is the \emph{Lobachevsky function}, see e.g.~\cite{thurston1979geometry}.

\begin{theorem}[\cite{Ken09}]
Let $\{ \gamma_n\}_{n\in \N}$ be a sequence of ASC. Suppose that
\ts $\frac{1}{n}\gamma_n$ \ts converges to a closed curve~$\gamma$
in $\mathbb{R}^3$ in the $\ell_{\infty}$ norm as $n \to \infty$. Then:
\[
\lim_{n\to \infty} \. \frac{1}{n^2} \. \log \numtilings_{\gamma_n} \, \to \, \Ups(g_{\max})\ts,
\]
where $g_{\max}: U \to R$ is the only extension of~$\gamma$
in \ts $\Lip_{[0,1]}$~that maximizes the following integral:
\[
\Ups(g) \, := \, \iint_{U} \sigma\bigl(\nabla g(x_1,x_2)\bigr) \ts dx_1\ts dx_2\ts,
\]
and $U$ is the region enclosed by the projection of~$\gamma$.
Moreover, for all $\e > 0$ the height function of a random tiling
chosen from the weighted measure associated to $w_n$ on height
functions with boundary $\gamma_n$, stays within $\e$ of $g_{\max}$
with probability $\to 1$ as $n\to\infty$.
\end{theorem}
	
The proof of this result is sketched in~\cite{Ken09} and is the analogue
of an earlier result for dominoes~\cite{CKP01}. The argument in the
latter paper extends to our setting of lozenges.

\smallskip

We are now ready to state
the variational principle for the weighted case. The proof is postponed to
Section~\ref{sec:proofwtvarprinciple}.
	
\begin{theorem}[Weighted variational principle] \label{thm:main}
Let $\{\gamma_n \}_{n\in \N}$ be a sequence of ASC, and let
$\{w_n\}_{n \in \N}$ be a  sequence of weight functions converging to
a function $\rho$.  Suppose that $\frac{1}{n}\gamma_n$ converges to a closed curved $\gamma$ in $\mathbb{R}^3$ in the
$\ell_{\infty}$ norm as $n\to \infty$. Then we have:
\[
		\lim_{n\to \infty }\. \frac{1}{n^2}\log\Zptn(H_{\gamma_n}, w_n) \. = \. \Psi(f_{\max})\ts.
\]
Here $f_{\max}: U \to \R$ is the only extension of $\gamma$  in $\Lip_{[0,1]}$ which
maximizes the following integral:
\begin{equation} \label{eq:Ww}
\Psi(f)\. := \, \iint_{U} \. \Bigl(\sigma(\nabla f)\ts +\ts L(x_1,x_2,\nabla f)\Bigr) \ts dx_1 \ts dx_2,
\end{equation}
where $U$ is the region enclosed by the projection of $\gamma$, and
\begin{equation}  \label{def:L}
L(x_1,x_2,\nabla f) \. :=\. \rho(x_1,x_2) \cdot (\partial_{x_1} f,\partial_{x_2} f,1-\partial_{x_1} f-\partial_{x_2} f).
\end{equation}
Moreover, for all $\e > 0$, the height function of a random tiling chosen from
the weighted measure associated to $w_n$ on height functions with boundary $\gamma_n$,
stays within $\e$ of $f_{\max}$ with probability tending to $1$ as
$n\to \infty$.
\end{theorem}
	
\bigskip

\section{From lozenge tilings to  standard Young tableaux} \label{s:count}

In this section we apply the weighted variational principle to prove the 
main result on asymptotics of the number of skew SYT of skew shapes with 
strongly stable shapes.

Recall that $\{ \nu_N = \lambda^{(N)}/\mu^{(N)}\}$ is a sequence of
skew shapes with the strongly stable shape $\psi/\phi$ as defined in
Section~\ref{def:stronglystableshape}. We denote by $U$ the region
enclosed by $\phi$ in the first quandrant of $\mathbb{R}^2$.

\subsection{The weight function of hook lengths} \label{sec:defptnZepZ}
In order to apply the weighted variational principle we need weight
functions that converge in the sense of Definition~\ref{def:convweight}.
In order to obtain a partition function that matches Naruse's
formula~\eqref{eq:NaruseTiling}, the natural choice of weight
function on $U$  is the following
$$
w_N(x,y) \, := \, \bigl(0,0,\log(\hk_{\lambda^{(N)}}(x,y)/\sqrt{N})\bigr).
$$
Denote by $\weight_N(h)$ the corresponding weight on height functions.
Then
\[
\weight(h) \. = \. (\sqrt{N})^{-|\mu^{(N)}|} \cdot \hookwts_{\lambda^{(N)}}(\tiling_h).
\]
However for this choice of weight function, $\log h_{\lambda^{(N)}}(x,y)$
can be very small for points $(x,y)$ near the border of the shape $\lambda^{(N)}$;
see Figure~\ref{f:pfskewSYT}:Left. In this regime, Property~\eqref{prop:defnweights}
might not hold. To fix this, we change the weight function to cap these small
values as follows. For $\e >0$ and $(x,y)$ in $\mathcal{C}(\psi/\phi)$, let
$$
w^{\e}_N(x,y) \. := \. \Bigl( 0, \. 0 , \, \max \bigl\{\log\bigl(\hk_{\lambda^{(N)}}(x,y)/\sqrt{N}\bigr), \log \e\bigr\} \Bigr).
$$
Denote by $\weight_N^{\e}(h)$ the corresponding weights on a height
function~$h$.  Similarly, denote by $Z_N$ and $Z_N^{\e}$ the corresponding partition
functions associated to weights $w_N$ and $w^{\e}_N$ respectively.

\begin{figure}
\includegraphics[scale=0.6]{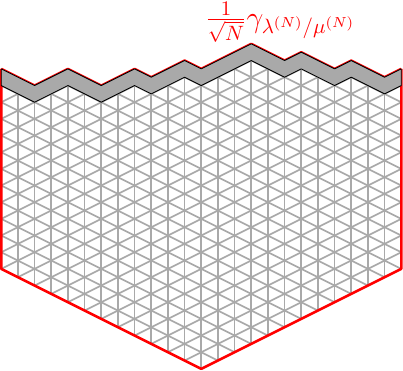} \quad \includegraphics[scale=0.6]{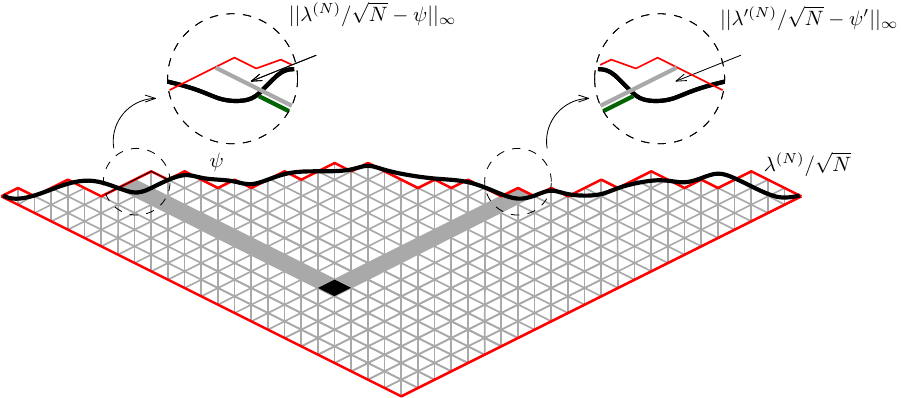}
\caption{Left: For points $(x,y)$ near the top border (schematically indicated in grey) of the region the values of $\log h_{\lambda}(x,y)$ are small and can affect convergence of the weight function. Right: The hook measured in $h_{\lambda^{(N)}}(x,y)$.}
\label{f:pfskewSYT}
\end{figure}

\medskip

\subsection{From lozenge tilings to counting tableaux}
We first show in Lemma~\ref{lemm:pfNaruse1} that the weighted variational principle, Theorem~\ref{thm:main},
applies to $Z_N^{\e}$. We then apply the variational principle in this
case to obtain
$$\lim_{N\to \infty} \. \frac{1}{N} \. \log Z_N^{\e} \. = \. c(\e),
$$
for some constant $c(\e)$ depending on $\e$ and the shapes $\psi$ and~$\phi$. We then show that
the constant $c(\e)$ converges to a constant $c$ as $\e \to 0$ (Lemma~\ref{lemm:pfNaruse3}).
In Corollary \ref{cor:logarithm} we show that
$\log Z_N^{\e}$ converges to $c$ as $\e \to 0$.
 Finally, we conclude
that
$$\lim_{N\to \infty} \. \frac{1}{N} \. \log Z_N \. = \. c,
$$
by showing that $\log Z_N^{\e}$ must also converge to $\lim_{N\to \infty} \. \frac{1}{N} \. \log Z_N$ as $\e \to 0$ (Theorem~\ref{cor:pfNaruse}).
In Section~\ref{sec:endpfskewthm}, we use this last result to prove
Theorem~\ref{thm:constantskew}. An explicit formulation of the constant $c$ is given in Corollary \ref{cor:explicit-constantskew}

We need some extra definitions. For the non-increasing left continuous function $\phi:[0,a] \to [0,b]$ define $\tilde{\phi}:[-b,a]\to \mathbb{R}$ by:
\[
\tilde{\phi}(u)\,=\, x+\phi(x) \quad \text{ for } x: x-\phi(x)=u.
\]
and define $\tilde{\psi}$ in an similar way. In other words, $\tilde{\phi}$ and $\tilde{\psi}$ are continuous analogues of the Russian representations of partitions. Next, we define the continuous limit of the ASCs associated to the sets of polygonal regions $H_{\lambda^{(N)}/\mu^{(N)}}$.

\begin{definition}\label{def:gamma}
Let  $\gamma$ in $\R^3$ be the continuous closed curve that is formed by: the line segment from $\{a,0,\tilde{\psi}(a)-\tilde{\phi}(a)\}$ to $\{a,0,0\}$, the line segment from $\{a,0,0\}$ to the origin, the line segment from the origin to $\{0,b,0\}$, the line segment from $\{0,b,0\}$ to $\{0,b,\tilde{\psi}(-b)-\tilde{\phi}(-b)\}$ and the curve $\{\phi(x),x,\tilde{\psi}(x-\phi(x))-\tilde{\phi}(x-\phi(x))\}$ for $x$ in $[0,a]$.
\end{definition}

Let $\Lip_{[0,1]}(\gamma)$ be the subset of $\Lip_{[0,1]}$ of functions which extend $\gamma$.

\begin{lemma} \label{lemm:pfNaruse1}
We have:
\[
\lim_{N\to \infty} \frac{1}{N} \log Z_N^{\e} \ = \, \sup_{f \in \Lip_{[0,1]}(\gamma)} \Psi_{\e}(f),
\]
where $\Psi_{\e}(\cdot)$ is the integral defined in~\eqref{eq:Ww} for the limiting weight function
$$
\rho_{\e}(x,y)\, :=  \Bigl( 0, \. 0, \, \max\bigl\{\log \hbar(x,y), \log \e \bigr\} \Bigr).
$$
\end{lemma}

\begin{proof}
First, we show that the weight function $w^{\e}_N(x,y)$  converges
to $\rho_{\e}(x,y)$, in the sense of
Definition~\ref{def:convweight} verifying property~\eqref{prop:defnweights}.
By convergence of the sequence of shapes, for $N$ large enough, either
both $\hk_{\lambda^{(N)}}(x,y)/\sqrt{N}$ and $\hbar(x,y)$ defined in \eqref{eq:defhbar} are smaller
than or equal $\e$ or both are greater or equal to~$\e$. In the first case,
we have \ts  $w^{\e}_N(x,y) = \rho_{\e}(x,y)=(0,0,\log\e)$,
and property~\eqref{prop:defnweights} vacuously holds.

In the second case we have that for all $(x,y)\in D$~: 
\begin{align*}
\left| w^{\e}_N(x\sqrt{N},y\sqrt{N}) - \rho_{\e}(x,y)\right| \, & =  \, \left| \log \frac{1}{\sqrt{N}} \. \hk_{\lambda^{(N)}}\bigl(\lfloor x
\sqrt{N}\rfloor,  \lfloor y
\sqrt{N}\rfloor\bigr)  - \log \hbar(x,y) \right| \\
&\leq \, k_{\e} \. \left|  \frac{1}{\sqrt{N}} \. \hk_{\lambda^{(N)}}\bigl(\lfloor x
\sqrt{N}\rfloor,  \lfloor y \sqrt{N}\rfloor \bigr)  - \hbar(x,y) \right|,
\end{align*}
where the inequality follows because the logarithm function is $k$-Lipschitz on $[\epsilon, \infty)$,
for some constant~$k_{\e}$.  From the definition of hook lengths
(see Figure~\ref{f:pfskewSYT}:Right), we also have:
\[
\left|  \frac{1}{\sqrt{N}} \. \hk_{\lambda^{(N)}}\bigl(\lfloor x \sqrt{N}\rfloor,  \lfloor y
\sqrt{N}\rfloor\bigr)  - \hbar(x,y) \right|  \leq
\bigl\|\lambda^{(N)}/\sqrt{N} - \psi \bigr\|_{\infty} + \bigl\|\lambda'^{(N)}/\sqrt{N} - \psi^{-1} \bigr\|_{\infty}\..
\]
Thus, by convergence of the sequence of shapes, we have:
\[
\lim_{N\to \infty} \. \left| w^{\e}_N(\sqrt{N}x,\sqrt{N}y) - \rho_{\e}(x,y)\right|
\, \leq  \, \lim_{N\to \infty} k_{\e} \cdot \left(
  \bigl\|\lambda^{(N)}/\sqrt{N} - \psi \bigr\|_{\infty} + \bigl\|\lambda'^{(N)}/\sqrt{N} - \psi^{-1} \bigr\|_{\infty} \right) \,= \. 0.
\]
This proves property~\eqref{prop:defnweights}.

By construction of the sequence of polygonal regions $H_{\lambda^{(N)}/\mu^{(N)}}$, after rescaling by $\sqrt{N}$, the corresponding sequence
of ASCs $\frac{1}{\sqrt{N}}\gamma_{\lambda^{(N)}/\mu^{(N)}}$ converges to $\gamma$ for the infinite norm.  Thus
the weighted variational principle, Theorem~\ref{thm:main},
applies giving
\begin{equation} \label{eq:pfconstant}
\lim_{N\to \infty} \. \frac{1}{N} \. \log Z_N^{\e} \, = \, \Psi_{\e}\bigl(f_{\max}\bigr),
\end{equation}
as desired.
\end{proof}

The next lemma explains what happens to  $\sup_{f \in \Lip_{[0,1](\gamma)}}\Psi_{\e}(f)$ as $\e
\to 0$.

\begin{lemma} \label{lemm:pfNaruse3}
Let $\Psi(f)$ be the integral defined in~\eqref{eq:Ww} for $f \in
\Lip_{[0,1]}(\gamma)$ and the weight function
$$
\rho(x,y)\, := \,  \left( 0, \. 0 \,, \log  \hbar(x,y)  \right)
$$
Then
$$\lim_{\e \to 0} \. \sup_{f \in \Lip_{[0,1]}(\gamma)} \bigl|\Psi_{\e}(f)\. -\. \Psi(f)\bigr| \, = \, 0\ts.
$$
In particular,
$$
\lim_{\e \to 0} \. \sup_{f \in \Lip_{[0,1]}(\gamma)} \Psi_{\e}(f)\,  = \, \sup_{f \in \Lip_{[0,1]}(\gamma)} \Psi(f)\..
$$
\end{lemma}

\begin{proof}
Let $f \in \Lip_{[0,1]}(\gamma)$, since the $||\cdot||_{\infty}$-norm of the partial derivatives of $f$ are
uniformly bounded by $1$ on $U$ (see \eqref{eq:defLipschitz}),  we obtain directly that:
$$
\aligned
\bigl|\Psi_{\e}(f)-\Psi(f)\bigr| \, & = \, \iint_{U} (\rho_{\e}(x_1,x_2)-\rho(x_1,x_2)) \cdot
(\partial_{x_1} f,\partial_{x_2}f,1-\partial_{x_1} f-\partial_{x_2} f) \. dx_1 \ts dx_2 \\
& \le \, \iint_{U} ||\rho_{\e}(x_1,x_2)-\rho(x_1,x_2)||_{\infty}\. dx_1\ts dx_2 \,
\endaligned
$$
The last integral of the previous inequality can we rewritten as:
\begin{eqnarray*}
\iint_{U} ||\rho_{\e}(x_1,x_2)-\rho(x_1,x_2)||_\infty \ts dx_1\ts dx_2 & = &
\iint_{\{\log \hbar(x_1,x_2) \leq \log \e \}} |\log \hbar(x_1,x_2)- \log {\e}| \ts dx_1\ts dx_2 \\
& \leq & \iint_{\{\log \hbar(x_1,x_2) \leq \log \e \}} \biggl(|\log \hbar(x_1,x_2)|+ |\log {\e}|\biggr)\ts \ts dx_1\ts dx_2\ts.
\end{eqnarray*}
The last integral converges to $0$ when $\e$ goes to $0$ provided the
function $\log \hbar$ is integrable on the domain $U$. To show the
integrability of $\log \hbar$ we use the following bounds on $\hbar(x,y)$:
\begin{equation} \label{eq:bound hbar}
|\phi(x)-y| \leq \hbar(x,y) \leq a+b,
\end{equation}
where the upper bound follows from the definition of a hook and the
bound of $\psi(x) \in [0,b]$. The lower bound follows from the
fact that $\hbar(x,y) \geq |\psi(x)-y|$
(see Figure~\ref{fig: bound hook}) and the inequality $\phi(x)\leq
\psi(x)$ implied by the definition of the stable shape
$\psi/\phi$. By taking the log and the absolute value, the bound
\eqref{eq:bound hbar} becomes
\begin{equation} \label{eq:bound hbar transformed}
|\log \hbar(x,y)| \leq \max( \log | \phi(x)-y|, 2\sqrt{2} b).
\end{equation}
Using standard analysis, one can show that the RHS of \eqref{eq:bound hbar transformed} is
integrable on  $U$ and thus the function $\log \hbar$ is
also integrable on $U$, as desired.
\end{proof}

\begin{figure}
\begin{center}
\includegraphics[scale=0.7]{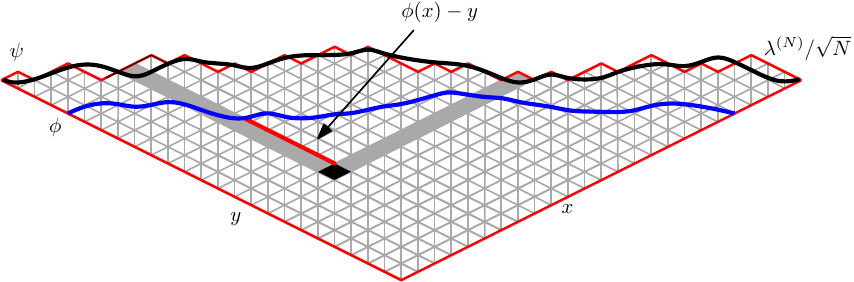}
\end{center}
\caption{Representation of the rescaled hook function
  $\hbar(x,y)$. The segment  of length $\phi(x)-y$ (depicted in red) is a lower bound for the hook length $\hbar(x,y)$ (depicted in gray).}
\label{fig: bound hook}
\end{figure}

Next, we write the log of the partition function $Z_N^{\e}$ in
terms of the log of the partition function $Z_N$.

As a direct corollary of the previous two lemmas we obtain the following.

\begin{corollary}\label{cor:logarithm}
	\[
	\lim_{\e \to 0} \lim_{N \to \infty} \frac{1}{N}\log Z^{\e}_N =  \sup_{f\in \Lip_{[0,1]}(\gamma)} \Psi(f).
	\]
\end{corollary}

\begin{proof}
This follows by combining lemmas~\ref{lemm:pfNaruse1} and \ref{lemm:pfNaruse3}.
\end{proof}

We can now prove our main tool used to prove the main result on
asymptotics of the number of skew SYT.

\begin{theorem} \label{cor:pfNaruse}
We have
\begin{equation}\label{eq:limlim}
\lim_{N \to \infty} \frac{1}{N}\log Z_N = \sup_{f\in \Lip_{[0,1]}(\gamma)} \Psi(f).
\end{equation}
\end{theorem}

\begin{proof}
By Corollary \ref{cor:logarithm} it suffices to show that
\[
 \lim_{N \to \infty} \frac{1}{N}\log Z_N = \lim_{\e \to 0} \lim_{N \to \infty} \frac{1}{N}\log Z^{\e}_N.
\]
Since by definition of $Z^{\e}_N$ and $Z_N$ (see \S~\ref{sec:defptnZepZ}) we
have that $Z^{\e}_N \geq Z_N$ and
\[
\frac{1}{N}\log Z^{\e}_N  \geq  \frac{1}{N}\log Z_N,
\]
it is enough to prove that
 \begin{equation}\label{eq:upperbound}
 \lim_{\e \to 0} \lim_{N \to \infty} \frac{1}{N}\log Z^{\e}_N \leq \lim_{N \to \infty} \frac{1}{N}\log Z_N.
 \end{equation}
By the median inequality we have
\begin{equation} \label{eq:pfmedian}
\frac{Z^{\e}_N}{Z_N} \, \leq \, \max_{h} \. \frac{ \weight^{\e}_N(h)}{\weight_N(h)}\..
\end{equation}
Outside of a border strip of $\mu^{(N)}$ of height $\lfloor \e \sqrt{N}\rfloor$ the weights will not change. The
hooks on the remaining lozenges in the strip are lower bounded by
their depth. So the RHS in \eqref{eq:pfmedian} can be bounded as follows,
\begin{equation} \label{eq:pf2median}
\max_{h} \frac{ \weight_N^{\e}(h)}{\weight_N(h)} \, \leq \, \frac{ (e^{\log \e})^{\e
N}}{ \prod_{k=1}^{\lfloor \e \sqrt{N}\rfloor} (e^{\log  \left( k/(\e \sqrt{N})\right) }  )^{\e \sqrt{N}}}\, = \,
\frac{ (e^{\log \e})^{\e N}}{ \exp\left(\sum_{k=1}^{\lfloor \e \sqrt{N}\rfloor} \e
\sqrt{N} \log \left(k/(\e \sqrt{N})\right) \right)}\..
\end{equation}
We can rewrite the denominator on the RHS above as
\begin{equation} \label{eq:pf3median}
\exp\left(\sum_{k=1}^{\lfloor \e \sqrt{N}\rfloor} \e \sqrt{N} \log
  \frac{k}{\e\sqrt{N}}\right) \,=\, \exp\left(\frac{\e N}{\e
    \sqrt{N}} \sum_{k=1}^{\lfloor \e \sqrt{N} \rfloor} \log \frac{k}{\e \sqrt{N}} \right).
\end{equation}
If $\e$ is fixed and $N$ goes to infinity, the quantity $\frac{1}{\e
	\sqrt{N}} \sum_{k=1}^{\lfloor \e \sqrt{N} \rfloor} \log \frac{k}{\e \sqrt{N}} $ is a Riemann sum for the function $x \to \log x$ on the interval $[0,1]$. Hence
\begin{equation}\label{eq:Riemann_sum}
\frac{1}{\e
	\sqrt{N}} \. \sum_{k=1}^{\lfloor \e \sqrt{N} \rfloor} \log \frac{k}{\e \sqrt{N}} \, =  \,  \int^{1}_0 \log x \ts dx \. + \. o_{\e}(1)
\end{equation}
 where $o_{\e}(1)$ is a function depending on $\e$ which goes to $0$ when $N$ goes to $\infty$.
Thus using \eqref{eq:pf3median} and \eqref{eq:Riemann_sum},  \eqref{eq:pf2median} becomes
\begin{equation} \label{eq:pfmedian-bound}
\max_{h} \frac{ \weight_N^{\e}(h)}{\weight_N(h)} \, \leq \,
\exp\left(\epsilon N\left(\log \epsilon - \int^1_0 \log x dx - \. o_{\e}(1)\right)\right).
\end{equation}
Combining the previous equation with the bound in \eqref{eq:pfmedian},
taking logs, and letting $N \to \infty$ yields the following upper bound
$$\lim_{N \to \infty} \frac{1}{N} \log \frac{Z^{\e}_N}{Z_N}  \leq \lim_{N \to \infty} \frac{1}{N} \log  \frac{ \weight_N^{\e}(h)}{\weight_N(h)}  \leq   \e \left(\log \e -\int^{1}_0 \log x \ts dx \right).$$

Since $\lim_{\e \to 0}   \e \left(\log \e -\int^{1}_0 \log x \ts dx \right)=0$, by taking the limit when
$\e$ goes to $0$ in the previous bound  we obtain \eqref{eq:upperbound} which finishes our proof.

\end{proof}

%
%
%
%

\medskip

\subsection{The number of standard Young tableaux} \label{sec:endpfskewthm}
We are now ready to prove Theorem~\ref{thm:constantskew}.  We require the following
technical result.

\begin{lemma} \label{lem:appskewSYT}  We have:
\[
\frac{1}{N} \left[\. \log \biggl(\. \sum_{h \in H_{\lambda^{(N)}/\mu^{(N)}}} \hookwts_{\lambda^{(N)}}(\tiling_h)\biggr)
 \. - \. \frac{\text{\rm area}(\phi)}{2} \ts N \log N\right] \, \to \, c,
\]
where $c:=\Psi(f_{\max})$ is a constant which depends only on $\psi$ and~$\phi$.
\end{lemma}

\begin{proof}
Recall that for the weight function $w_N(x,y)$ and a height function $h$ in $H_{\lambda^{(N)}/\mu^{(N)}}$ we
have that
\begin{align}
\notag
\hookwts_{\lambda^{(N)}}(\tiling_h) &= \prod_{\lozenge \in \tiling_h} \hk_{\lambda^{(N)}}(x_{\lozenge},y_{\lozenge})
 = (\sqrt{N})^{|\mu^{(N)}|} \times \prod_{\lozenge \in \tiling_h}
e^{w^{i_{\lozenge}}(x_{\lozenge}y_{\lozenge})} \\
& =
(\sqrt{N})^{|\mu^{(N)}|} \times \weight(h),
\label{eq:hk2weight}
\end{align}
where $\weight(h)$ is a defined in~\eqref{eq:defweight-tiling}. Then the log of the partition function of all height functions in $H_{\lambda^{(N)}/\mu^{(N)}}$ equals
\begin{equation}
\label{eq:hk2weightsum}
\log \Bigl(\sum_{h \in H_{\lambda^{(N)}/\mu^{(N)}}} \.
\hookwts_{\lambda^{(N)}}(\tiling_h)\Bigr) \,  = \, \log (\sqrt{N})^{|\mu^{(N)}|} \. + \. \log Z_N,
\end{equation}
where $Z_N = \sum_{h \in
  H_{\lambda^{(N)}/\mu^{(N)}}} \weight(h)$. We treat each of the two summands in the RHS above separately.

By condition \eqref{eq:condition_area} on the area of $\phi$ in the definition of the stable shape we have that
\begin{equation} \label{eq:pfleadterm}
 \log (\sqrt{N})^{|\mu^{(N)}|} \,  = \, \frac{1}{2} \bigl|\mu^{(N)}\bigr| \log N \, = \,
 \frac{\area(\phi)}{2} \ts N\log N \. + \. o(N).
\end{equation}
Next, by Theorem~\ref{cor:pfNaruse} we have
\begin{equation} \label{eq:pfconstant1}
\lim_{N\to \infty} \frac{1}{N} Z_N \. = \. c,
\end{equation}
where $c := \Psi(f_{\max})$ is a constant that only depends on $\psi$ and $\phi$.

Finally, we take the limit as
$N \to \infty$ in  \eqref{eq:hk2weightsum} and use both
\eqref{eq:pfleadterm} and \eqref{eq:pfconstant1} to obtain the desired result.
\end{proof}

\begin{proof}[Proof of Theorem~\ref{thm:constantskew}]
	We take logs in \eqref{eq:NaruseTiling} to obtain
	\begin{equation} \label{eq:initialpf}
	\log f^{\nu^{(N)}} = \log |\nu^{(N)}|! - \left(\sum_{(x,y) \in \lambda^{(N)}}
        \log\hk_{\lambda^{(N)}}(x,y)\right) + \log \left(\sum_{h \in
            H_{\lambda^{(N)}/\mu^{(N)}}} \hookwts_{\lambda}(\tiling_h)\right)
	\end{equation}
	Observe that $|\nu^{(N)}| = N + O(\sqrt{N})$ as $N\to
        \infty$. Then by
        Stirling's formula we have
\begin{equation} \label{eq:appStirling}
\log |\nu^{(N)}|! = N\log N - N + O(\sqrt{N}\log N).
\end{equation}
 Next, we use the definition and compactness of the stable shape $\mathcal{C}(\psi)$
\[
\sum_{(x,y) \in \lambda^{(N)}} \log \hk_{\lambda^{(N)}}(x,y) \, = \,
N\iint_{\mathcal{C}(\psi)} \. \log\bigl(\sqrt{N}\hbar(x,y)\bigr) \ts dx\ts dy \. + \. o(N),
\]
where the leading $N$ outside the integral comes from a change of variables $x \to \sqrt{N}x$,
$y\to \sqrt{N}y$ and the $\sqrt{N}$ inside the integral comes from
rewriting $\hk_{\lambda^{(N)}}(\cdot ,\cdot)$ in terms of $\hbar(x,y)$
defined in \eqref{eq:defhbar}. The error term $o(N)$ comes from
approximating the sum with the scaled integral (cf. \cite[Thm. 6.3]{MPP4}).

By linearity of integration with respect to the integrand $\frac{1}{2}\log N + \log \hbar(x,y)$ we obtain
	\begin{equation} \label{eq:loghookslam}
	\sum_{(x,y) \in \lambda^{(N)}} \log \hk_{\lambda^{(N)}}(x,y) \, = \,
        \frac{\area(\psi)}{2} \ts N \ts \log N \. + \. k(\psi)N \. + \. o(N),
	\end{equation}
where \ts $k(\psi) = \iint_{\mathcal{C}(\psi)} \hbar(x,y) \. dx\ts dy$.
Lastly, applying to each term in~\eqref{eq:initialpf} the bounds from
\eqref{eq:appStirling}, \eqref{eq:loghookslam} and
Lemma~\ref{lem:appskewSYT} respectively we obtain
\[
\log f^{\nu^{(N)}} \, =\, \left(1-\frac{\area(\psi/\phi)}{2}\right) \. N \ts \log N \. + \. c(\psi/\phi)N \. + \. o(N),
\]
where $c(\psi/\phi):=c+k(\psi)$ is the sum of the constant  $c$ from
Lemma~\ref{lem:appskewSYT} and $k(\psi)$. Finally, since \ts $\area(\psi/\phi)=1$, the result follows.
\end{proof}

We end this section by extracting from the proof above the explicit
expression for the  constant of
Theorem~\ref{thm:constantskew}.

\begin{corollary} \label{cor:explicit-constantskew}
The constant $c(\psi/\phi)$ of Theorem~\ref{thm:constantskew} is given by
\[
c(\psi/\phi) \,:=\, k(\psi) \. + \. \Psi(f_{\max}),
\]
where
\begin{align*}
k(\psi) \. & = \.\iint_{\mathcal{C}(\psi)} \hbar(x,y) \. dx\ts dy,\\
\Psi(f) \. &= \. \iint_U \Bigl( \sigma(\nabla f)
\. + \. (1-\partial_{x} f - \partial_{y} f) \ts \log \hbar(x,y) \Bigr) \ts dx\ts dy,
\end{align*}
for $\sigma(\cdot)$ defined by~\eqref{eq:Lob}, $U$ is the
region enclosed by the projection of $\gamma$ from Definition~\ref{def:gamma}, and $f_{\max}:U\to \mathbb{R}$ is the only function in $\Lip_{[0,1]}(\gamma)$ which maximizes the integral $\Psi(f)$.
\end{corollary}

\bigskip

\section{Proof of the weighted variational principle} \label{sec:proofwtvarprinciple}
Our strategy to prove this theorem consists of three parts. In the first part
we give a lemma (Lemma~\ref{entropy_fundamental_domain}) that shows that
fundamental domains with similar plane-like boundary conditions have the
same number of tilings and that all those tilings contain a similar number
of lozenges of each type. Both numbers depend on the slope of the domain.
In the second part we give a lemma (Lemma~\ref{lemma:weighted}) that shows
that the weighted contribution of lozenges with similar plane-like boundary
conditions is also the same. Finally, in the third part we use the two
previous lemmas to prove the weighted variational principle.

\medskip

\subsection{Tilings of similar plane-like regions (unweighted)}
Let $(s,t)$ be a pair of numbers such that $\{ 0 \leq s,t,1-s-t\leq 1 \}$, let $\epsilon> 0$
and let $D_m$ be the $m \times m$ diamond of the hexagonal grid whose left
corner is the origin.
		
Let	$\overline{H}^{\e}_{D_m}(s,t)$ be the set of admissible boundary
height functions $\bar{h}:  \partial D_m \rightarrow \Z $, such that:
		\begin{itemize}
			\item the left corner of the diamond has height $0$
			\item for all $x=(x_1,x_2) \in \partial D_m $ we have
			$$| \bar{h}(x_1,x_2) -( sx_1+tx_2 ) | \leq \e m.$$
		\end{itemize}

	\begin{lemma}\label{entropy_fundamental_domain}
		Let $(s,t)$ be such that $\{ 0 \leq s,t,1-s-t\leq 1 \}$, let $\epsilon
		> 0$ and let $D_m\subset \mathbb{Z}^2$ and
                $\overline{H}^{\e}_{D_m}(s,t)$ be as defined above
              . Then
		for each $\bar{h} \in \overline{H}^{\e}_{D_m}(s,t)$ we
                have that
\begin{equation}
		\label{eq:lem1item1}
\lim_{m \to \infty} \frac{1}{m^2} \log \numtilings(\bar{h}) = \lim_{m \to
			\infty} \frac{1}{m^2} \log \sum_{\bar{h}
			\in\overline{H}^{\e}_{D_m}(s,t)} N(\bar{h}) =
                      \sigma(s,t)+ O (\e \log(1/\e)),
\end{equation}
and
\begin{equation}
		\label{eq:lem1item2}
		\lim_{m \to \infty} \frac{1}{m^2} \left(\log \numlozenge^{(1)}(\bar{h}),
		\log \numlozenge^{(2)}(\bar{h}),\log
                \numlozenge^{(3)}(\bar{h})\right) =
		(s,t,1-s-t) + O(\e){\bf 1}.
\end{equation}
	\end{lemma}
	
		\begin{figure}
		\includegraphics[scale=1]{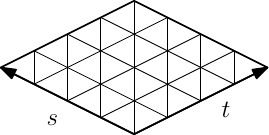}
		\caption{The slope of a periodic tiling.}
		\label{fig:slope}
		\end{figure}

	\begin{proof}
		Let  $\mathcal{P}_m(s,t)$ be the set of tilings of $D_m$ with periodic
		boundary conditions with slope $(s,t)$ and ${N}_m(s,t)$ be the
		number of tilings in $\mathcal{P}_m(s,t)$. Note that
		$\mathcal{P}_m(s,t)$  is also the set of tilings of a torus with
		slope $(s,t)$. By \cite[Thm.~8]{Ken09} we have that:
		\[
		\frac{1}{m^2}\ts \log N_m(s,t) \. = \. \sigma(s,t)\ts +\ts o(1),
		\]
		and that each of those tilings has exactly $\{m^2s,m^2t,m^2(1-s-t)\}$ lozenges of each type.
		Additionally, if we choose a height function uniformly amongst all height functions in $\mathcal{P}_m(s,t)$ then we have the following concentration results:
		\begin{equation} \label{eq:concentration}
		\mathbb{P}{\Big (}|h(x_1,x_2)-(sx_1+tx_2)| \ts \geq \ts \e{\Big)} \, \leq \, e^{4\e m}.
		\end{equation}
This can be shown by applying the same martingale
argument as in \cite[Prop. 22]{CEP96}. Although the
argument in this paper is made for simply connected
regions, it extends for tilings of a torus with given slopes.

Denote by $\mathcal{P}^{\e}_m(s,t)$ the set of periodic configurations on a torus of size $m$ whose height function stays within $\e m$ of a linear plane of slope $(s,t)$ that is:
$$\mathcal{P}^{\e}_m(s,t) \, := \, \left\{ \ts h \in \mathcal{P}_m(s,t) ~:~\max_{x \in D_m} \bigl\{|h(x_1,x_2)-(sx_1+tx_2)|\bigr\} \geq \e m\right\}.
$$
Let $N^{\e}_m(s,t)$ be the size of $\mathcal{P}^{\e}_m(s,t)$.
As a direct consequence of the inequality~\eqref{eq:concentration}, we have:
\[
\frac{1}{m^2} \log \left( N_m(s,t)(1-e^{-c\e m}) \right) \leq  \frac{1}{m^2} \log N^{\e}_m(s,t) \leq \frac{1}{m^2} \log N_m(s,t)
\]
Therefore,
\[
\lim_{m \to \infty} \frac{1}{m^2}\log N^{\e}_m(s,t) \. = \. \sigma(s,t).
\]
We must now distinguish between the case where \ts
$\e\leq \frac{1}{2}(1-\max\{s,t,1-s-t\})$  and the case
$\e >\frac{1}{2}\max\{s,t,1-s-t\}$.

\smallskip

\noindent {\em Case~1:} \ts Suppose \ts $ \e  \leq \frac{1}{2}(1-\max\{s,t,1-s-t\})$.
Consider \ts $\bar{h}  \in \overline{H}^{\e}_{D_{m}}(s,t)$ \ts and \ts $h_- \in
\mathcal{P}^{\e}_{m(1-3\e)}(s,t)$. For all $x=(x_1,x_2) \in \partial
D_{m(1-3\e)}$ and $y=(y_1,y_2) \in \partial D_m$  we have
		\begin{eqnarray*}
			\lefteqn{\bar{h}(y)-h_-(x) \, \leq }\\
& \, \leq \,  &  \biggl[\bar{h}(y)-(sy_1+ty_2)\biggr] \,+\, \biggl[(sy_1
         +ty_2)-(sx_1+tx_2)\biggr]  \,+\, \biggl[(sx_1+tx_2)-h_-(x)\biggr] \\
			& \, \leq \, & \e \ts m \. + \. \max\{s,t,1-s-t\}\cdot \max\{y_1-x_1,y_2-x_2\} \.+ \.\e \ts m \\
			& \, \leq \,  &  {\bigl(}1-\max\{s,t,1-s-t\}{\bigr)}\|x-y\|_1 \. + \.\max\{s,t,1-s-t\}\cdot \max\{y_1-x_1,y_2-x_2\} \\
			& \, \leq \, & \max\{y_1-x_1,y_2-x_2\}.\\
		\end{eqnarray*}
		
\nin	
where $\|\cdot\|_1$ denotes the $1$-norm.

\begin{figure}[hbt]
			\includegraphics[scale=0.7]{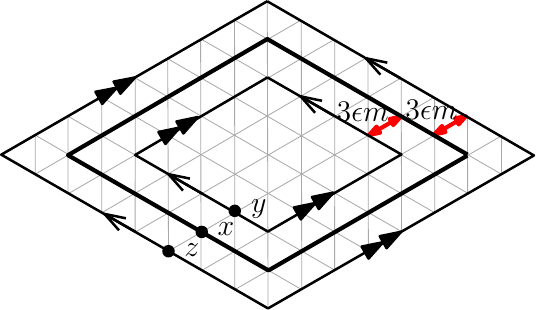}
			\caption{Illustration of the proof of Lemma~\ref{entropy_fundamental_domain}. The number of
				tilings with boundary conditions in $\overline{h} \in \partial D_m$
				is at least the number of tilings with periodic boundary conditions
				in $\partial D_{m(1-3\epsilon)}$ and at most the number of tilings
				with periodic boundary conditions in
                                $\partial D_{m(1+3\epsilon)}$.}
			\label{fig:lemma1}
		\end{figure}

Using Lemma~\ref{lem:extension},
we deduce that there exist a height function $h$ on $D_m$  such that
$h=\bar{h}$ on $\partial{D_m}$ and $h=h_-$ on $\partial{D_{m(1-3\e)}}$.
As a consequence, we obtain that  $N(\bar{h}) \geq N^{\e}_{m(1-3\e)}(s,t)$.
For the same reasons, for
$h_+ \in \mathcal{P}^{\e}_{m(1+3\e)}(s,t)$, for all $x \in \partial D_m$
and $z \in \partial D_{m(1+3\e)}$ we have:
\[
|\bar{h}(x)-h_+(z)| \. \leq  \. \min\{z_1-x_1,z_2-x_2\}.
\]
Thus, every boundary height functions in \ts $\mathcal{P}^{\e}_{m(1+3\e)}(s,t)$ \ts
can be extended to $\bar{h}$ on $\partial D_m$. This implies:
\[
N^{\e}_{m(1-3\e)}(s,t) \. \leq \. N(\bar{h}) \. \leq \. N^{\e}_{m(1+3\e)}(s,t),
\]
which can be rewritten as
\[
\frac{1}{m^2} \log N^{\e}_{m(1-3\e)}(s,t) \, \leq \,\frac{1}{m^2} \.
\log N(\bar{h}) \, \leq \, \frac{1}{m^2} \. \log N^{\e}_{m(1+3\e)}(s,t).
\]
Since \ts $1/(m^2(1-3\e)) = 1/m^2+O(\e)$, we deduce that
$$\lim_{m \to \infty} \. \frac{1}{m^2} \.\log N(\bar{h}) \, = \, \sigma(s,t) \. + \. O(\e).
$$
Finally, we can bound the number of boundary conditions in
$H^{\e}_{D_m}(s,t)$ by the number of different types of lozenges to
the power of the length of $\partial D_m$. Since there are at most $3^{4m}=
e^{o(m^2)}$ different boundary height functions in
$\overline{H}^{\e}_{D_m}(s,t)$, then this allows us to deduce~\eqref{eq:lem1item1}:
\[
\lim_{m \to \infty} \frac{1}{m^2} \log  \sum_{\bar{h}  \in\overline{H}^{\e}_{D_m}(s,t)} N(\bar{h}) = \sigma(s,t)+ O (\e).
\]
For the second part of the statement, we notice that when attaching two tilings as described
above (see Figure~\ref{fig:lemma1}), we are adding at most $\e^2m^2$ tilings of each type. Hence we obtain that for all $i \in \{1,2,3\}$:
\[
\mathcal{L}^{(i)}(h_-)-\e^2m^2 \leq  \mathcal{L}^{(i)}(\bar{h})\leq \mathcal{L}^{(i)}(h_+)+\e^2m^2.
\]
Dividing by $m^2$ and taking the logarithm, we obtain~\eqref{eq:lem1item2}.
		
\medskip

\noindent {\em Case 2:} \ts Suppose \ts $ \e  \geq
\frac{1}{2}(1-\max\{s,t,1-s-t\})$.  Let \ts $\bar{h} \in
\overline{H}^{\e}_{D_{m}}(s,t)$. Without loss of generality we can assume
that \ts $\max\{s,t,1-s-t\}=1-s-t$ so that \ts $ \e  \geq (s+t)/2$.
The height difference between the top vertex and the bottom
vertex of each vertical section of $D_m$ is at most \ts $4\ts \e \ts m$.
Hence, each of such vertical sections contains at most \ts
$\lfloor 4 \e m \rfloor$ vertical edges. This means that the total number of
        non-horizontal lozenges in each tiling of a height function
        that extends $ \bar{h}$ is
        smaller than \ts $\lfloor 4\ts \epsilon \ts m^2\rfloor$ \ts and implies directly (\ref{eq:lem1item2}). Notice that we can determine a tiling
        by specifying what is the position of the non-horizontal lozenges and their
        types. Hence the total number of tilings $N(\bar{h})$ is bounded by
        ${m^2 \choose \lfloor 4\e m^2 \rfloor}2^{\lfloor 4\ts \epsilon \ts m^2 \rfloor}$. By using Stirling's formula, we obtain
\[
N(\bar{h}) \leq {m^2 \choose \lfloor 4\e m^2 \rfloor}2^{\lfloor 4\ts \epsilon \ts m^2\rfloor}
\. = \. e^{m^2O(\e \log(1/\e))}\ts.
\]
Therefore, the total number of configurations with boundary $\bar{h}$ satisfies
		\[
		\frac{1}{m^2}\. N(\bar{h})\, = \, O\bigl(\epsilon \log(1/\e) \bigr)\. + \.o(1).
		\]
		Since $\sigma(0,0)=0$, this implies (\ref{eq:lem1item1}) and concludes our proof.
	\end{proof}

	Lemma~\ref{entropy_fundamental_domain} holds when we replace
        lozenges by equilateral triangles. This will be useful for the
        remainder of the proof as explained in Section~\ref{subsec:weighted}.
	
	\begin {corollary} \label{cor:mainlemmaTriangles}
Let $T_m$ be an equilateral triangle of size $m$ and
        $\overline{H}^{\e}_{T_m}(s,t)$ be as defined above. Then
        for each $\bar{h} \in \overline{H}^{\e}_{T_m}(s,t)$ we
        have that
        \begin{equation}
        \label{eq:lem1item1triangles}
        \lim_{m \to \infty} \frac{4}{\sqrt{3}m^2} \log N(\bar{h}) \, = \, \lim_{m \to
        	\infty} \.  \frac{4}{\sqrt{3}m^2} \. \log \left(\sum_{\bar{h}
        	\in\overline{H}^{\e}_{D_m}(s,t)} N(\bar{h})\right) \, = \,
        \sigma(s,t) \. + \. O\bigl(\e \log(1/\e)\bigr),
        \end{equation}
and
        \begin{equation}
        \label{eq:lem1item2triangles}
        \lim_{m \to \infty} \frac{4}{\sqrt{3}m^2} \left(\log \mathcal{L}^{(1)}(\bar{h}),
        \log \mathcal{L}^{(2)}(\bar{h}),\log
        \mathcal{L}^{(3)}(\bar{h})\right) \, = \,
        (s,t,1-s-t) \. + \. O(\e){\bf 1}.
        \end{equation}
\end{corollary}

\begin{proof}
	Let $T_m$ be a triangle of size $m$ and $\bar{h}$ be a
        boundary height function which stays within $\e m$ of the
        plane with slope $(s,t)$. For each $h \in \bar{h}$, if we
        reflect $h$ along one side we obtain a height function of a
        lozenge $D_m$ which also stays within  $\e m$ of the plane
        with slope $(s,t)$. Hence we can bound $N(\bar{h})^2$ by the number of ways to extend a boundary in $\overline{H}^{\e}_{D_m}(s,t)$ and we obtain:
$$
\frac{2}{m^2}\log N(\bar{h}) \, = \, \frac{1}{m^2}\log\left(N(\bar{h})^2\right) \, \leq \,
\frac{1}{m^2}\.\log \left[\sum_{\bar{h}  \in\overline{H}^{\e}_{D_m}(s,t)} N(\bar{h})\right]
\. \leq \. \sigma(s,t)\. + \. O\bigl(\e \log(1/\e)\bigr).
$$
Now consider a triangle $T_{m^2}$ of size $m^2$, $\bar{h}$ be a boundary
height function which stays within $\e m$ of the plane with slope $(s,t)$.
We can fill partially $T_{m^2}$ with $m-o(1)$ lozenges of size $m$ each
having the same periodic boundary height function with slope $(s,t)$.
Using a similar argument as the one in the previous lemma for attaching
configurations, we can attach $\bar{h}$ to the height function on those
lozenges and we obtain:
$$
\sigma(s,t)\. +\. O\bigl(\e \log (1/\e)\bigr) \,\leq \, \frac{1}{2m^2}\log N(\bar{h}),
$$
as desired. \end{proof}

\bigskip

\subsection{Tilings of similar plane-like regions (weighted)} \label{subsec:weighted}
For the remainder of this proof we will be working with triangles
since later in this proof we will need to approximate surfaces with
piecewise-linear functions. Such approximations are done in a standard way using
triangles (see for example Lemma 2.2 in \cite{CKP01}).

Since the weight of each individual lozenge tile depends on its position in the lattice,
we now evaluate the weight contribution of a large triangle as a
function of its position.

\begin{lemma} \label{lemma:weighted}
	Let $x=(x_1,x_2) \in \R^2$ and $\ell \in \R$ be such that
        $\rho$ is smooth on $B(x,\ell)$. Let  $T(x,\ell n)$ be the
        triangle of size $ \ell n$ centered at the point
        $x^n:=(\lfloor n x_1 \rfloor,\lfloor n x_2 \rfloor)$ and  let
        $\bar{h} \in \overline{H}^{\e}_{T(x, \ell n)}(s,t)$. For a converging sequence of weights $\{w_{n} \}_{n \in \mathbb{N}}$ we have :
	\begin{align}
	\lim_{n\to \infty} \frac{4}{\sqrt{3}(\ell n)^2} \. \log
	Z(H_{\bar{h}},w_{\ell}) \, &  = \, \lim_{n\to
		\infty}\frac{4}{\sqrt{3}(\ell n)^2}\. \log \left[\sum_{\bar{h} \in
		\overline{H}^{\e}_{T(x,\ell n)}(s,t)} \. Z\bigl(H_{\bar{h}},w_{\ell}\bigr)\right] \notag \\
	&= \, \sigma(s,t)\. + \. L(x_1,x_2,s,t)\. + \. O\bigl(\e \log (1/\e)\bigr)\. + \. O(\ell),  \label{eq1:lemmaweighted}
	\end{align}
	where $Z(H_{\bar{h}},w_{\ell})$ is the total weight of all configurations
	with boundary $\bar{h}$.
\end{lemma}

\begin{proof}
The sequence of weights $\{w_{n} \}_{n \in \mathbb{N}}$
is convergent, by Condition~(ii) of Definition~\ref{def:convweight}.  Thus,
for all \ts $ny \in T(x,\ell n)$, and for each type of lozenge tile $i \in \{1,2.3\}$,
we have:
\[
|w^i_n(ny)-\rho^i(x)| \leq |w^i_n(ny)-\rho^i(y)|+|\rho^i(y)-\rho^i(x)| =o(1)+O(\ell).
\]
Here we used the smoothness of $\rho$ on $B(x,\ell)$ to bound $|\rho^i(y)-\rho^i(x)| = O(\ell)$.
	This means that for all height function $h \in
	H^{\e}_{T(x,l)}(s,t)$ with boundary $\bar{h}$, we must have:
\begin{align*}
	\weight(h) & =  \prod_{\lozenge \in h} e^{w^{i_{\lozenge}}(x_{\lozenge})} =  \prod_{\lozenge \in h} e^{(\rho^{i_\lozenge}(x)+O(\ell)+o(1))}\\
	& =  \, \sum_{j=1}^3 \. \exp\left[\frac{\sqrt{3}(\ell n)^2}{4}\bigl(x_j\. + \. o(1)\. + \. O(\e)\bigr)\right] \cdot
\exp\biggl[\rho^{1}(x)\. + \. o(1)\. + \. O(\ell)\biggr] \\
	& = \, \exp\left[\frac{\sqrt{3}(\ell n)^2}{4}\biggl(L(x_1,x_2,s,t)\. + \. o(1)\. + \. O(\ell)\. + \.  O(\e)\biggr)\right], \label{eq:oneheight}
\end{align*}
	where $x_3 = 1-x_1-x_2$. Then the contribution
	of all configurations with boundary $\bar{h}$ is
	given by:
	\begin{align*}
	Z(H_{\bar{h}},w_{n}) & \, = \, \sum_{h \in H^{\e}_{T(x,\ell)}(s,t)} \weight(h) \\
	& = \, \exp\left[\frac{\sqrt{3}(\ell \ts n)^2}{4}\biggl(L(x_1,x_2,s,t)\. + \. o(1) \. + \. O(\ell) \. + \. O(\e )\biggr)\right] \. N(\bar{h}).
	\end{align*}
Applying Corollary~\ref{cor:mainlemmaTriangles} to the equation above, we obtain:
	\begin{align*}
	Z(H_{\bar{h}},w_{n}) & \, = \, \exp\left[\frac{\sqrt{3}(\ell \ts n)^2}{4}\biggl(L(x_1,x_2,s,t) \. + \. o(1) \. + \. O(\e)\biggr)\right] \, \times \notag \\
&
\hskip2.cm \times \, \exp\left[\frac{\sqrt{3}(\ell \ts n)^2}{4}\biggl(\sigma(s,t) \. + \. o(1) \. + \. O\bigl(\e
  \log (1/\e)\bigr)\biggr)\right]  \notag \\
	&  \, =  \,  \exp\left[\frac{\sqrt{3}(\ell\ts n)^2}{4}\biggl(\sigma(s,t) \. + \. L(x_1,x_2,s,t) \. + \. o(1) \. + \. O(\ell) \. + \. O\bigl(\e\log(1/\e)\bigr)\biggr)\right].
	\end{align*}
	Then \eqref{eq1:lemmaweighted} follows by taking the logarithm
        of the equation above. Since the number of boundary height
        functions for a given triangle is bounded by $3^{3\ell
          n}=e^{o(\ell^2 n^2)}$, we also obtain:
\[
	\lim_{\ell\to \infty}\frac{4}{\sqrt{3}(\ell n)^2}
	\sum_{\bar{h} \in
		\overline{H}^{\e}_{T(x,\ell)}(s,t)} Z(H_{\bar{h}},w_{\ell}) \, = \,
    \sigma(s,t)\. +\. L(x_1,x_2,s,t)\. +\. O\bigl(\e \log (1/\e)\bigr)\. +\. O(\ell).
\]
This finishes the proof.
\end{proof}

\medskip

\subsection{Proof of Theorem~\ref{thm:main}}

We now prove the weighted variational principle. At this stage our
strategy is exactly the same as Theorem 4.3 in \cite{CKP01} or
Theorem 2.9 in \cite{MT}. We recall the following two lemmas from   \cite{CKP01} which will be useful in our proof.

\begin{lemma}[\cite{CKP01} Lemma 2.2] \label{pf:linear}
	For $\ell > 0$, consider a mesh made up of equilateral triangles of side length
	$\ell$ (which we call an $\ell$-mesh). Let $f \in \Lip_{[0,1]}$  be such that
    $f= \gamma$ on $U$, and let $\e > 0$. If $\ell$ is sufficiently small then
    on at least $(1-\e)$ fractions of the triangles in the $\ell$-mesh that
    intersect $U$ we have the following two properties:
	\begin{enumerate}
		\item The piecewise linear approximation $\wt{f}$ stays within $\ell \e$ of $f$.
		\item For at least a $(1-\e)$  fraction (in measure) of the points $x$ of the triangle,
        the tilt $\nabla f(x)$ exists and is within~$\e$  of $\wt{f}(x)$.
	\end{enumerate}
\end{lemma}

\begin{lemma}[\cite{CKP01} Lemma 2.3] \label{pf:linearlem2fromCKP}
	Suppose that $T$ is an equilateral triangle of length $\ell$, and the
	height function $f$ satisfies $|f_{\delta T} - \wt{f}| \leq \e
	\ell$ on $\delta T$, where $\wt{f}$ is the piecewise linear
	approximation from Lemma~\ref{pf:linear}, then
	\[
	\Psi(f) \. = \. \Psi(\wt{f}) \. + \. o\bigl(\text{\rm area}(T)\bigr).
	\]
\end{lemma}

\begin{remark}  {\rm
	Note that Lemma~2.3 in \cite{CKP01} is stated for $\Ups(\cdot)$,
	however since  the function $L(\cdots)$ from the definition of
        $\Psi(\cdot)$ in \eqref{eq:Ww} is
	linear, this lemma still holds for $\Psi(\cdot)$.
}\end{remark}

Next, we  approximate the partition function of all height
functions which stays close to a given function $f$ using
$\Psi$ and show that the error term goes to zero.

\begin{lemma}
	Let $f \in \Lip_{[0,1]}$  be such that $f= \gamma$ on $U$ and let $\delta > 0$.  If we denote by $
	\Zptn(H^{\delta}_{f}, w_n)$ the total weight of height functions
	which stay within $\delta n$  of $f$, then:
	\begin{equation} \label{eq:pflb}
	\lim_{n \to \infty} \frac{1}{n^2}\log \Zptn(H_{f}^{\delta}, w_n) \, =\, \Psi(f) \. + \. o(1).
	\end{equation}
\end{lemma}

\begin{proof}
Let $\e < 1$ and consider a sequence of grids $\{G_n\}_{n \in \N}$
which partition  the triangular lattice into equilateral
triangles of size $\ell n$.  Denote by $\wt{G}_n$ the
$\ell$-mesh obtained by rescaling $G_n$ and $\wt{f}$
the linear approximation of $f$ on $\wt{G}_n$. See
Figure~\ref{fig:pfvarprin}.
	
We start by approximating $\Psi(f)$ by the terms on the RHS of~\eqref{eq1:lemmaweighted}
of Lemma~\ref{lemma:weighted}. According to Lemma \ref{pf:linear}, for $\ell$ small enough we have:
	\[
	\sup_{x \in T}|f(x)-\wt{f}(x)| \leq \ell\e
	\]
on all but a portion at most $\e$ of the triangles in $U$. Next, we rewrite $\Psi(\wt{f})$ as
\begin{equation} \label{eq:Psialpha2sumPsialpha}
	\Psi(\wt{f}) \, = \, \sum_{T \in U} \frac{4\ell^2}{\sqrt{3}}   \left(
	\sigma(\nabla \wt{f})
	+L(x^{T}_1,x^{T}_2,\wt{f}) \right) \. + \. o\bigl(\area(U)\bigr),
\end{equation}
where the error term $o\bigl(\area(U)\bigr)$ comes from bounding the following integral
	\[
	\sum_{T\in U} \. \iint_T \left(\rho(x_1,x_2) - \rho(x_1^T,x_2^T)\right)\cdot (\partial_{x_1} f, \partial_{x_2}
	f, 1 - \partial_{x_1}f-\partial_{x_2}f) \. d x_1 \ts d x_2,
	\]
using the uniform continuity of $\rho$ on each component of $U$ where it is smooth.

Combining \eqref{eq:Psialpha2sumPsialpha} with Lemma~\ref{pf:linearlem2fromCKP} we have:
$$
\left|\Psi(f)-\sum_{T \in U} \frac{4\ell^2}{\sqrt{3}} \left(\sigma(\nabla \wt{f})
	\ts + \ts L(x^{T}_1,x^{T}_2,\wt{f}) \right)\right| \. = \. o(1).
$$
With this approximation, we are now ready to prove \eqref{eq:pflb}.
Choose  $\delta < \ell\e$ and $\{h_n\}$ to be any height function
with boundary~$\gamma_n$. Define
$$
a_n(x) \, := \, \min \bigl\{h_n(x),\lfloor nf(x/n)+\delta\rfloor\bigr\}
\quad \text{and} \quad g_n(x) \, := \, \max\bigl\{ a_n(x), \ts \lfloor nf(x/n)-\delta\rfloor\bigr\}\ts.
$$
Then \ts $|g_n/n-f| \leq \delta/2$  and $g_n \in  H^{\delta}_{f}$, by construction.
		
\begin{figure}
		\includegraphics[scale=0.57]{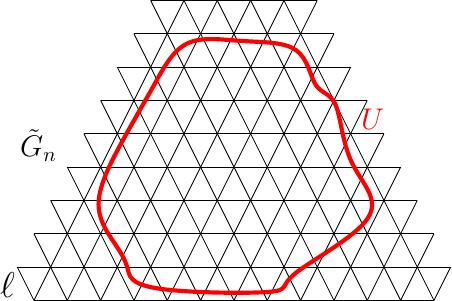}
		\caption{The grid $\wt{G}_n$ of equilateral triangles of size $\ell$
			that partitions $U$ used in the proof of Theorem~\ref{thm:main}.}
		\label{fig:pfvarprin}
\end{figure}
	
We can ignore  the contribution of triangles that are not fully included in $U$  which is $O(\delta)=O(\e)$.
The quantity  $\log \Zptn(H_{f}^{\delta}, w_n)$ is bounded from below by the weight
of all height functions that agree with $g_n$ on the boundary of all triangles in $G_n$ completely
contained in~$U$ after rescaling.  This gives:
\begin{align*}
	\log \prod_{T \in	U} 	Z(H_{\bar{g}_{\partial 	T}},w_n) \, = \, \sum_{T \in	U}
	\log  Z(H_{\bar{g}_{\partial T}},w_n) \, \leq \, \log Z(H_{f}^{\delta}, w_n),
\end{align*}
where the product is taken over all triangle $T$ fully contained in~$U$ and such that \[
\sup_{x \in T}|f(x)-\wt{f}(x)| \leq \frac{\delta}{2}.
\] Now include
the $O(\e)$ in the bound.  Then \ts $\log Z(H_{f}^{\delta}, w_n)$ \ts is bounded from above
by the total free product of all height functions which stays within $\delta$
of $\wt{f}$ on each one of those triangles. In other words,
$$
	\log
	Z(H_{f}^{\delta}, w_n) \, \leq \, \log \prod_{T \in	U} \.
	Z(H^{\delta}_{\wt{f}_{\partial T}},w_n) \, =  \, \sum_{T \in	U}
	\log  Z(H^{\delta}_{\wt{f}_{\partial T}},w_n) \. + \.O(\e).
$$
Using Lemma~\ref{lemma:weighted} to approximate each $\log Z(H_{\bar{g}_{\partial T}},w_n)$
and  $\log Z(H^{\delta}_{\wt{f}_{\partial T}},w_n) $ in the above inequalities, we obtain:
\begin{align*}
 &\frac{1}{n^2} \log Z(H^{\delta}_{f}, w_n) \, = 	\\
& \qquad = \, \frac{1}{n^2} \sum_{T \in U}
\left[\frac{\sqrt{3}}{4}\ell^2n^{2} \left(\sigma(\wt{f}) + L(x^{T}_1,x^{T}_2,\wt{f})
+ o( 1)+O(\ell) +O(\e \log 1 / \e)\right)\right]+o(1) \\
& \qquad = \, \sum_{T \in U} \left[\frac{\sqrt{3}}{4}\ell^2
\left( \sigma(\wt{f}) +L(x^{T}_1,x^{T}_2,\wt{f}) + o( 1)+O(\ell) \ts +\ts O(\e \log 1 / \e)\right)\right]+o(1)\\
& \qquad = \, \Psi(\wt{f})+O(\ell) \. + \. O(\e \log 1 / \e) \. + \. o(1).
\end{align*}
Since  both $\ell$ and $\e$ can be chosen as small as needed when
 $\delta\to 0$, we have:
$$ \frac{1}{n^2} \log Z(H^{\delta}_{f}, w_n) \, = 	\,
\Psi(\wt{f})\. + \. O(\ell) \. + \. O(\e \log 1 / \e) \. + \. o(1) \, = \, \Psi(f) \. + \. o(1),
$$
as desired.
\end{proof}

The function $\sigma$ is strictly convex and $L$ is linear,
thus the function $\sigma + L$ is itself strictly convex.
This implies that there exist a unique function $f_{\max}$
in $\Lip_{[0,1]}$ that maximize $\Psi$. By the previous lemma, we obtain:
$$
\lim_{n \to \infty} \. \frac{1}{n^2}\. \log \Zptn(H_{\gamma_n}, w_n) \, \geq\, \Psi(f_{\max}).
$$
Moreover, the set $\Lip_{[0,1]}(U,\gamma)$ of functions $f \in \Lip_{[0,1]}$ such that $f= \gamma$ on~$U$
is compact for the $\ell_{\infty}$ norm. Hence, for every fixed $\alpha > 0$,
there exist a finite covering $\bigcup_{i=1}^{k_{\alpha}}B_{\ell_{\infty}}(f_i,\delta_{i})$ of $\Lip_{[0,1]}(U,\gamma)$ such that for all $i \leq k_{\alpha}$:
\begin{equation}
\lim_{n \to \infty} \frac{1}{n^2}\log \Zptn(H_{f_i}^{\delta_{i}}, w_n) \, \leq \Psi(f_i) +\alpha .
\end{equation}
If we denote by $C(\delta, \alpha)$ the number of balls in this covering, this implies that for all $\delta, \alpha > 0$:
$$
\lim_{n \to \infty} \. \frac{1}{n^2}\. \log \Zptn(H_{\gamma_n}, w_n) \, \leq \, \Psi(f_{\max}) + \lim_{n \to \infty} \frac{1}{n^2}\. \log C(\delta, \alpha)\,+\alpha \, = \, \Psi(f_{\max})+ \alpha.
$$
Letting $\delta$ and $\alpha$ go to $0$ gives the desired result. This finishes the proof of Theorem~\ref{thm:main}.

\bigskip

\section{Final remarks and open problems}\label{s:finrem}

\subsection{}
There are other positive formulas for $f^{\lambda/\mu}$ using
the Littlewood--Richardson coefficients and the \emph{Okounkov--Olshanski
formula}, see \cite[\S 9]{MPP1} and \cite{MZ} for the discussion and references.
It would be interesting to see if the variational principle applies in
either case.

\subsection{} \label{ss:finrem-racah}
In case of the thick hooks (see~$\S$\ref{ss:intro-hooks}),
the variational principle result (Theorem~\ref{thm:main})
is already interesting and is now  well understood.
It corresponds to a degenerate case of more general
weights introduced in~\cite{BGR} and further studied
in~\cite{Betea,DK} (see also~\cite{MPP3}), where both
the frozen region and the probability density are computed.

It is worth comparing frozen regions in the uniform and weighed cases,
see Figure~\ref{f:lozenge-hexagon}.  The uniform frozen region is
famously a circle, while the weighted frozen region is an algebraic
curve with only mirror symmetry.  Let us mention that explicit product
formulas for \emph{$q$-Racah polynomials} allow a direct sampling from
these weighted tilings in this case, see~\cite[$\S$7.5]{Betea}
and~\cite[$\S$9]{BGR}.  This approach does not generalize to other
skew shapes.

\begin{figure}[hbt]
\includegraphics[scale=0.4]{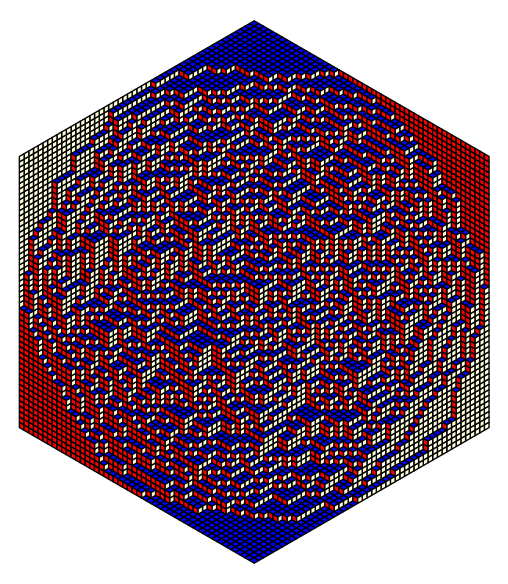} \hskip1.4cm \includegraphics[scale=0.4]{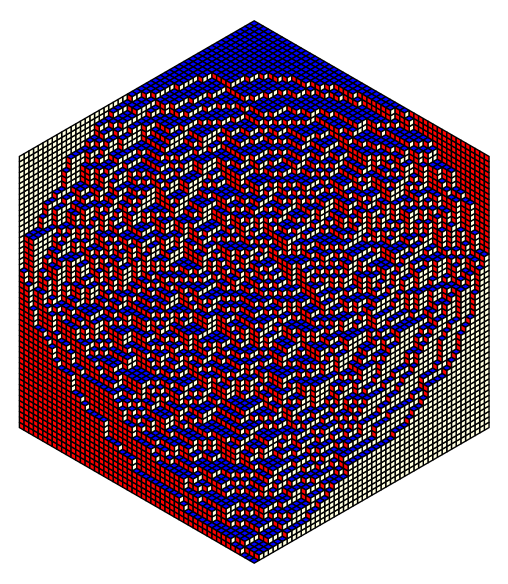}
\caption{Uniform and weighted random lozenge tilings of the hexagon \ts
  $\Hex(50,50,50)$ from \cite[Fig. 2]{MPP3}.}
\label{f:lozenge-hexagon}
\end{figure}

\subsection{}
It would be interesting to compute the frozen region explicitly
for the weighted lozenge tilings in some important special cases,
such as thick ribbons described in
$\S$\ref{ss:intro-ribbons}.  From the variational principle we
cannot even tell if these regions are bounded by algebraic curves.

\subsection{}
Beside stable limits shapes, there are other asymptotic regimes
when the problem of computing $f^{\la/\mu}$ is of interest,
see~\cite{DF,MPP4,Stanley_skewSYT}. Except for the case when
$|\mu|=O(1)$, obtaining better bounds is an interesting and difficult
challenge.

\subsection{} \label{ss:finrem-var}
In an important recent development, Sun showed the existence of limit curves
for random standard Young tableaux with stable limit shape~\cite{Sun},
also by modifying the variational principle and using the \emph{beads model} 
to encode standard tableaux. Soon after Gordenko \cite{Gor} wrote an explicit 
conjectural description of this limit shape using a modified TASEP on dimer 
configurations. This suggests that in principle one can hope to evaluate
the constant~$C$ of the for thick ribbons (see $\S$\ref{ss:intro-ribbons}).

\subsection{} \label{ss:finrem-non}
It would be interesting to see if the strategy sketched in~\cite[3.5]{Pak-ICM} 
can be used to conclude that there is no natural bijective proof of the Naruse 
hook-length formula~\eqref{eq:Naruse}.
Let us mention that~\cite{Kon} gives a bijective proof of a recurrence
involved in the proof of the NHLF. Unfortunately, there seem to be no
way to use this bijection for uniform sampling of random standard
Young tableaux of skew shape.

\vskip.7cm

\subsection*{Acknowledgements} The authors would like to thank
Nishant Chandgotia, Jehanne Dousse, Valentin F\'eray,
Anna Gordenko, Vadim Gorin, Rick Kenyon, Victor Kleptsyn,
Georg Menz, Jay Pantone, Leo Petrov, Istv\'an Prause,
Dan Romik, Richard Stanley
and Damir Yeliussizov for helpful discussions.  We are especially
thankful to Greta Panova for the ongoing extended collaboration
in the area; notably, the idea of this paper is stemming from
the lozenge tiling results in~\cite{MPP3}. We also thank the anonymous
referee for helpful comments and suggestions.
The first two authors were partially supported by the NSF.


\vskip.9cm

\bibliographystyle{alpha}

\begin{thebibliography}{CKP01}


\bibitem[AdR]{AdR}
R.~Adin and Y.~Roichman, Standard {Y}oung tableaux, in
{\em Handbook of enumerative combinatorics}, 
CRC Press, Boca Raton, FL, 2015, 895--974.


\bibitem[AHRV]{AHRV}
O.~Angel, A.~E.~Holroyd, D.~Romik and B.~Vir\'ag,
Random sorting networks, \emph{Adv.\ Math.}~\textbf{215}
(2007), 839--868.

\bibitem[AsR]{AsR}
R.~A.~Askey and R.~Roy, Barnes $G$-function, in
\emph{NIST Handbook of Mathematical Functions}
(F.~W.~J.~Olver et al., eds.), Cambridge Univ.\ Press, 2012.

\bibitem[Bet]{Betea}
D.~Betea, \emph{Elliptic Combinatorics and Markov Processes},
Ph.D.\ thesis, Caltech, 2012, 116~pp.; available at
\ts \href{http://thesis.library.caltech.edu/7115/1/betea_thesis.pdf}{https://tinyurl.com/yy37sjvb}

\bibitem[BGR]{BGR}
A.~Borodin, V.~Gorin and E.~M.~Rains,
{$q$}-distributions on boxed plane partitions,
{\em Selecta Math.}~{\bf 16} (2010), 731--789.

\bibitem[CPT]{CPT}
N.~Chandgotia, I.~Pak and M.~Tassy,
Kirszbraun--type theorems for graphs,
\emph{J.\ Combin.\ Theory, Ser.~B} \textbf{137} (2019), 10--24.

\bibitem[CEP]{CEP96}
H.~Cohn, N.~Elkies and J.~Propp,
Local statistics for random domino tilings of the {A}ztec diamond,
{\em Duke Math.~J.}~\textbf{85} (1996), 117--166.

\bibitem[CKP]{CKP01}
H.~Cohn, R.~Kenyon and J.~Propp,
A variational principle for domino tilings.
{\em J.~AMS}~\textbf{14} (2001), 297--346.

\bibitem[Dau]{Dau}
D.~Dauvergne,
The Archimedean limit of random sorting networks, preprint (2018), 61~pp.;
\ts
{\tt arXiv:} {\tt 1802.08934}.

\bibitem[DK]{DK}
E.~Dimitrov and A.~Knizel,
Log-gases on quadratic lattices via discrete loop equations
and $q$-boxed plane partitions,
\emph{J.\ Funct.\ Anal.}~\textbf{276} (2019), 3067--3169.

\bibitem[DF]{DF}
J.~Dousse and V.~F\'eray,
Asymptotics for skew standard Young tableaux via bounds for
characters, \emph{Proc.\ AMS}~\textbf{147} (2019), 4189--4203.


\bibitem[Feit]{Feit}
W.~Feit,
The degree formula for the skew-representations of the symmetric group,
\emph{Proc.~AMS}~\textbf{4} (1953), 740--744.

\bibitem[FeS]{FeS}
 V.~F\'eray and P.~\'Sniady,
Asymptotics of characters of symmetric groups related to
Stanley character formula, {\em Ann.\ of Math.}~\textbf{173}
(2011), 887--906.

\bibitem[Gor]{Gor} A.~Gordenko, Limit shapes of large skew Young tableaux
and a modification of the TASEP process, preprint (2020), 43~pp.;
\ts {\tt  arXiv:2009.10480}.

\bibitem[Ken]{Ken09}
R.~Kenyon, Lectures on dimers, in \emph{Statistical mechanics},
AMS, Providence, RI, 2009, 191--230.

\bibitem[Kon]{Kon}
M.~Konvalinka, A bijective proof of the hook-length formula for skew
shapes, \emph{European\ J.~Combin.}~\textbf{88} (2020), 103104, 14~pp.


\bibitem[MT]{MT}
G.~Menz and M.~Tassy, A variational principle for a
non-integrable model, \emph{Probab.\ Theory Related
Fields}~\textbf{177} (2020), 747--822.

\bibitem[MPP1]{MPP1}
A.~H.~Morales, I.~Pak and G.~Panova,
Hook formulas for skew shapes I. $q$-analogues and bijections,
\emph{J.~Combin. Theory, Ser.~A}~\textbf{154} (2018), 350--405.

\bibitem[MPP2]{MPP2}
A.~H.~Morales, I.~Pak and G.~Panova,
Hook formulas for skew shapes II. Combinatorial proofs and enumerative applications,
\emph{SIAM Jour.\ Discrete Math.}~\textbf{31} (2017), 1953--1989.

\bibitem[MPP3]{MPP3}
A.~H.~Morales, I.~Pak and G.~Panova,
Hook formulas for skew shapes III. Multivariate and product formulas,
\emph{Algebraic Combinatorics}~\textbf{2} (2019), 815--861.

\bibitem[MPP4]{MPP4}
A.~H.~Morales, I.~Pak and G.~Panova,
Asymptotics of the number of standard Young tableaux of skew shape,
\emph{European\ J.~Combin}~\textbf{70} (2018), 26--49.

\bibitem[MZ]{MZ}
A.~H.~Morales, and D.~G.~Zhu, 
On the Okounkov-Olshanski formula for standard tableaux of skew shapes, preprint (2020), 37~pp.;
\ts {\tt arXiv:2007.05006}

\bibitem[Nar]{Nar}
H.~Naruse, {S}chubert calculus and hook formula, talk slides at
\emph{73rd {S}\'em.\ {L}othar.\ {C}ombin.},
Strobl, Austria, 2014; available at \ts
\href{https://www.emis.de/journals/SLC/wpapers/s73vortrag/naruse.pdf}{https://tinyurl.com/y5zunwk4}

\bibitem[NO]{NO}
H.~Naruse and S.~Okada,
Skew hook formula for $d$-complete posets,
\emph{Algebraic Combinatorics}~\textbf{2} (2019), 541--571.

\bibitem[Pak1]{Pak-ICM}
I.~Pak, Complexity problems in enumerative combinatorics,
in \emph{Proc.\ ICM Rio de Janeiro}, Vol.~IV. Invited lectures, World Sci., Hackensack, NJ, 2018, 3153--3180;
expanded version available at \ts {\tt arXiv:1803.06636}.

\bibitem[Pak2]{Pak-skew}
I.~Pak, Skew shape asymptotics, a case-based introduction,
preprint (2020), 19~pp.; available at \ts
\href{https://www.math.ucla.edu/~pak/papers/Ribbon1.pdf}{https://tinyurl.com/y54opehb}

\bibitem[PST]{PST}
I.~Pak, A.~Sheffer, and M.~Tassy, Fast domino tileability,
{\em Discrete Comput.\ Geom.}~\textbf{56} (2016), 377--394.

\bibitem[PR]{PR}
B.~Pittel and D.~Romik,
Limit shapes for random square Young tableaux,
\emph{Adv.\ Appl.\ Math.}~\textbf{38} (2007), 164--209.


\bibitem[Rom]{Rom}
D.~Romik,
\emph{The surprising mathematics of longest increasing subsequences},
Cambridge Univ.~Press, New York, 2015.

\bibitem[Sta1]{Stanley_skewSYT}
R.~P.~Stanley, On the enumeration of skew {Y}oung tableaux,
{\em Adv.\ Appl.\ Math.}~\textbf{30} (2003), 283--294.


\bibitem[Sta2]{Stanley-EC}
R.~P.~Stanley, {\em Enumerative Combinatorics}, vol.~1
and~2,
Cambridge Univ.~Press, 2012 and~1999.

\bibitem[Sun]{Sun}
W.~Sun, Dimer model, bead and standard Young tableaux:
finite cases and limit shapes, preprint (2018), 67 pp.~;
\ts {\tt arXiv:1804.03414}.

\bibitem[Thu]{Thu}
W.~P.~Thurston,
Groups, tilings and finite state automata,
in \emph{Lecture Notes}, AMS Summer Meetings, Bolder, CO,
1989, 51~pp.; available at \ts
\href{http://timo.jolivet.free.fr/docs/ThurstonLectNotes.pdf}{https://tinyurl.com/y62mzl4r}

\bibitem[TM]{thurston1979geometry}
W.~P. Thurston and J.~W. Milnor,
{\em The geometry and topology of three-manifolds},
Princeton Univ.\ Press, Princeton, NJ, 1979.

\bibitem[MW]{}
M.~Wachs, Flagged {S}chur functions, {S}chubert polynomials, and symmetrizing operators, {\em J. Combin. Theory,
Ser. A}~\textbf{40} (1985), 276--289.

\vskip.7cm
\end{thebibliography}

\end{document}